\documentclass[a4paper,oneside,11pt]{scrartcl}
\input{mypaper.sty}
\usepackage{mathrsfs}
\usepackage{enumerate}
\usepackage{subcaption}

\def\T {\mathcal{T}}
\def\Pa{\mathcal{P}}
\def\J {\mathcal{J}}
\def\D      {\mathscr{D}}

\def\dellipse{\mathcal{E}}
\def\tx      {\tilde{x}}
\def\ty      {\tilde{y}}
\def\tz      {\tilde{z}}
\def\tH      {\tilde{H}}
\def\tk      {\tilde{k}}

\def\det   {\text{det}}

\begin{document}
\title{Fractional Laplacian -- Quadrature rules for singular double integrals in 3D}
\author{B. Feist and M. Bebendorf\footnote{Mathematisches Institut, Universit\"at Bayreuth, 95447 Bayreuth, Germany}}
\date{\today}
\maketitle

%\tableofcontents
%\newpage
\begin{abstract}
In this article, quadrature rules for the efficient computation of the stiffness matrix for the fractional Laplacian in three dimensions
are presented. These rules are based on the Duffy transformation, which is a common tool for singularity removal.
Here, this transformation is adapted to the needs of the fractional Laplacian in three dimensions.
The integrals resulting from this Duffy transformation are regular integrals over less-dimensional domains.
We present bounds on the number of Gauss points to guarantee error estimates which are of the same order of
magnitude as the finite element error. The methods presented in this article can easily be adapted to other singular double integrals in three dimensions with algebraic singularities.
\end{abstract}
\section{Introduction }\label{sec:intro}
The fraction Laplacian is used in many applications, e.g.\ in the modeling of cardiac electrical propagation \cite{Bueno-Orovio}, in fractional quantum mechanics \cite{Laskin_qunatum}, \cite{Laskin_schroedinger}, in the modeling of anomalous diffusion \cite{Carafelli_Diff}, \cite{Zaslavyky}, and in fluid mechanics \cite{Constantin06} and \cite{Constantin16}.
Assuming that $\Omega \subset \R^d$ is a Lipschitz domain and $f \in H^r(\Omega)$, $r\geq-s$,
then we consider the fractional Poisson problem
\begin{equation}\label{eq:fracL}
\begin{aligned} 
   (-\Delta)^s u &= f \quad\text{in } \Omega, \\
               u &= 0 \quad\text{on } \R^d \backslash \Omega,
\end{aligned}
\end{equation}
where the \textit{fractional Laplacian} (see~\cite{AcBo2017}) is defined as
\[
   (-\Delta)^s u (x) = c_{d,s}\, p.v. \int_{\R^d} \frac{u(x)-u(y)}{\lvert x-y\rvert^{d+2s}}\ud y, \quad  c_{d,s} := \frac{2^{2s}\,\Gamma(s+d/2)}{\pi^{d/2}\,\Gamma(1-s)}.
\]
Here, $s\in(0,1)$ is called the order of the fractional Laplacian, $\Gamma$ is the Gamma function, and p.v.\ denotes the Cauchy principal value of the integral.

For the numerical treatment of \eqref{eq:fracL}, the fractional Laplacian is discretized via a usual Galerkin
method. The coefficient matrix of the resulting linear system consists of entries for which reliable and efficient
quadrature rules need to be developed. In particular, the strong singularity of the arising double integrals
needs to be accounted for. In two dimensions ($d=2$) such quadrature rules have been developed in \cite{AiCl2018}.
In this article, quadrature rules are derived for the three-dimensional ($d=3$) case.

Notice that the fractional Laplacian requires the evaluation of an integral over an unbounded domain although $u$ vanishes outside~$\Omega$.
Nevertheless, the weak formulation of \eqref{eq:fracL} can be done in such a way that only integrals over $\Omega$ and
its boundary~$\partial\Omega$ appear. Quadrature rules for the
treatment of singular integrals usually rely on the Duffy transformation~\cite{Duffy}. We apply this method to
the interaction of two tetrahedra and the interaction of tetrahedra with triangles. Both cases arise after
discretization of~$\Omega$ and~$\partial\Omega$.

The article is organized as follows. In Sect.~\ref{sec:fl}, we introduce the variational formulation of \eqref{eq:fracL}
together with its reformulation over bounded domains. In order to be able to adapt the accuracy of the
developed quadrature rules to the finite element error, we provide a finite element error estimate presented in~\cite{AiCl2018}.
Sect.~\ref{sec:Duffy_Trafo} is devoted to the Duffy transformation. We present an approach for singular integrals
over multi-dimensional domains. In order to lift the singularity, it will be seen that a certain ordering
of the variables of integration is required. In Sect.~\ref{sec:Duffy_3d3d}, the interaction of two tetrahedra
is studied. Four singular cases have to be distinguished: a common vertex, a common edge, a common face, and identical tetrahedra.
For each of these cases, it is required to split up the domain of integration to ensure a suitable ordering of the integration variables.
Depending on the singularity, multiple splits have to be applied.
Corresponding cases are investigated in Sect.~\ref{sec:Duffy_3d2d},
where the interaction of a tetrahedron and a triangle is treated. As a result of Sect.~\ref{sec:Duffy_Trafo},
we obtain regular integrals over less-dimensional domains of integration. For the latter, in Sect.~\ref{sec:Error_Estimates},
suitable tensor Gauss quadrature rules are presented. Their error estimates rely on the analytic extension of
the integrands. We show that the integrands arising from the Duffy transformations developed in Sect.~\ref{sec:Duffy_Trafo}
satisfy the requirements of these error estimates. This allows to adapt the number of
Gauss points to the order of magnitude of the finite element error.
The last section presents two numerical tests. First, the accuracy of the quadrature rules is verified in the
case of the interaction of two tetrahedra and the interaction of a tetrahedron with a triangle. The second numerical
test is focused at the numerical solution of the fractional boundary value problem~\eqref{eq:fracL}. Here, all mentioned singularity cases appear.
The computed solution is compared with the exact solution known for a particular choice of the right-hand side~$f$ and the computational domain~$\Omega$.

The quadrature rules developed in this article for the fractional Laplacian can also be applied to other singular double integrals
over two tetrahedra or over the combination of a tetrahedron and a triangle in three dimensions.

\section{Fractional Laplacian}\label{sec:fl}
The solution of the problem~\eqref{eq:fracL} is searched for in the Sobolev space
\[
 H^s(\Omega) = \{ v \in L^2(\Omega): |v|_{H^s(\Omega)} < \infty \},
\]
where
\begin{align*}
   |v|_{H^s(\Omega)}^2 =  \int_\Omega\int_\Omega \frac{[v(x)-v(y)]^2}{|x-y|^{d+2s}}\ud x \ud y
\end{align*}
denotes the Slobodeckij seminorm. The space $H^s(\Omega)$ is a Hilbert space, equipped with the norm
\begin{align*}
   \norm{v}_{H^s(\Omega)} = \norm{v}_{L^2(\Omega)} + |v|_{H^s(\Omega)}.
\end{align*}
Zero trace spaces $H_0^s(\Omega)$ can be defined as the closure of $C_0^\infty(\Omega)$ with respect to this norm.

Due to the non-local nature of the operator, we need to define
\[
   \tH^s(\Omega) =\{ u \in H^s(\R^d):u=0\text{ in }\R^d\setminus\Omega\}.
\]
$\tH^s(\Omega)$ is the closure of $C_0^\infty(\Omega)$ in $H^s(\R^d)$; see~\cite[Chap.~3]{McLean}.
It is known (see \cite{ag2017_2}) that $\tH^s(\Omega)$ coincides with $H_0^s(\Omega)$ for $s \neq 1/2$,
and for $s=1/2$ it holds that $\tilde H^{1/2}(\Omega) \subset H_0^{1/2}(\Omega)$.
Negative order spaces $H^{-s}(\Omega)$ can be defined as the dual space of~$\tH^s(\Omega)$. 

With this notation we can reformulate \eqref{eq:fracL} in a weak form: given $f \in H^{r}(\Omega)$ 
find $u\in \tilde H^s(\Omega)$ satisfying
\begin{equation*}
   a(u,v) = \langle f,v\rangle_{\Omega}, \quad v \in \tH^s(\Omega).
\end{equation*}
By using symmetry and the Gauss theorem (see~\cite{AiCl2018}), the bilinear form can be simplified to
\[
   a(u,v) =  \frac{c_{d,s}}{2} \int_\Omega \int_\Omega \frac{[u(x) - u(y)]\,[v(x)-v(y)]}{ |x-y|^{d+2s}}\ud x \ud y
           + \frac{c_{d,s}}{2s} \int_\Omega u(x)\,v(x) \int_{\partial \Omega} \frac{(y-x)^T\, n_y}{|x-y|^{d+2s}}\ud s_y \ud x,
\]
where $n_y$ is the outer normal vector with respect to $\partial \Omega$ and $\langle\cdot,\cdot\rangle_\Omega$ denotes the duality pairing.
This simplification is crucial, since the domain of integration is now bounded.
Moreover, $\tilde H^s(\Omega)$ can be equipped with the energy norm
\[
  \lVert u \rVert_{\tilde H^s(\Omega)} =  \sqrt{a(u,u)} = \sqrt{\frac{c_{d,s}}{2}}\,|u|_{H^s(\R^d)}.
\]

Let the set $\{\varphi_{1},\dots, \varphi_{N}\}$ denote the nodal basis of the space of piecewise linear functions~$V(\mathcal{T})\subset \tilde H^s(\Omega)$,
where $\mathcal{T}$ is a regular partition  of $\Omega$ into $M$ tetrahedra and $N$ inner points.
$\mathcal{P}$ is the set of panels, being the faces of the tetrahedra of $\T$.
Additionally, we impose two assumptions on the discretization.
The first one is that the elements of the discretization, the panels and the tetrahedra, are shape-regular, i.e.\ the interior angles of the tetrahedra are bounded below by the angle $\theta_\T>0$: 
\begin{equation}\label{eq:theta_T}
	\theta_t > \theta_\T \quad \text{ for all } t \in \T,    
\end{equation}
where $\theta_t$ is minimal interior angle of the tetrahedron $t$.
Since each element of $\Pa$ is a face of a tetrahedron in $\T$, the interior angles of the panels are also bounded below by $\theta_\T$.
The second assumption is that the tetrahedra are quasi-uniform, i.e.\ there exist constants $c_1, c_2 >0$ with
\begin{equation} \label{assum:quasi_uniform}
    c_1\, h_{t_1} \leq h_{t_2} \leq c_2\, h_{t_1}
\end{equation}
for all $t_1,t_2 \in \T$, where $h_t$ is defined as the diameter of the tetrahedron $t$. 

Moreover, we introduce $h$ as the mean of all diameters, i.e.\
\begin{equation*}
    h := \frac{1}{|\T|} \sum_{t\in \T} h_t,
\end{equation*}
then \eqref{assum:quasi_uniform} is also valid for $h$, i.e.\
\begin{equation}\label{assum:quasi_uniform_h}
   c_1\, h \leq h_t \leq c_2 \, h
\end{equation}
for all $t \in \T$.  

The Galerkin method applied to \eqref{eq:fracL} yields the discrete fractional Poisson problem $Ax=g$ with $A \in \R^{N \times N}$, $g \in \R^N$ having the entries
\begin{equation*}
 \begin{split}
  a_{ij} &=\quad   \frac{c_{d,s}}{2} \int_\Omega \int_\Omega \frac{[\varphi_i(x) - \varphi_i(y)]\,[\varphi_j(x)-\varphi_j(y)]}{ |x-y|^{d+2s}}\ud x\ud y \\
         &\quad  + \frac{c_{d,s}}{2s} \int_\Omega \fie_i(x)\,\fie_j(x) \int_{\partial \Omega} \frac{(y-x)^T\, n_y}{|x-y|^{d+2s}}\ud s_y \ud x, \quad i,j = 1,\dots,N, \\
  g_{i} &= \langle f,\varphi_i\rangle_{\Omega}, \quad i = 1,\dots, N.
 \end{split}
\end{equation*}
By using the compact support of the linear basis functions, the expression for $a_{ij}$ can be simplified to
\begin{align}\label{eq:aij_nah}
\begin{split}
  a_{ij} &=\quad   \frac{c_{d,s}}{2} \int_{\Omega_{ij}} \int_{\Omega_{ij}} \frac{[\varphi_i(x) - \varphi_i(y)]\,[\varphi_j(x)-\varphi_j(y)]}{ |x-y|^{d+2s}}\ud x\ud y \\
         &\quad  + \frac{c_{d,s}}{2s} \int_{\Omega_{ij}} \fie_i(x)\,\fie_j(x) \int_{\partial\Omega_{ij}} \frac{(y-x)^T\, n_y}{|x-y|^{d+2s}}\ud s_y \ud x 
\end{split}         
\end{align}
where $\Omega_{ij} := \suppa\fie_i \cup \suppa\fie_j$. 
If the supports of the basis functions $\fie_i$ and $\fie_j$ are disjoint, 
the computation of the entry~$a_{ij}$ simplifies to
\begin{equation} \label{eq:aij_fern}
   a_{ij} = -c_{d,s} \int_{\suppa\fie_i} \int_{\suppa\fie_j} \frac{\fie_i(x)\, \fie_j(y)}{ |x-y|^{d+2s}}\ud x\ud y.
\end{equation}
For the following error estimates on the Galerkin approximation $u_h$ see \cite{AiCl2018}.
\begin{theorem} \label{thm:error_u_h}
 If the family of triangulations $\T_h$ is shape-regular and globally quasi-uniform, 
 and $u\in H^l(\Omega)$ for $0<s<l<1$ or $1/2 < s <1$ and $1<l<2$, then
 \begin{equation}\label{eq:error_u_h}
     \norm{u-u_h}_{\tilde H^s(\Omega)} \leq C(s,d)\, h^{l-s}\, |u|_{H^l(\Omega)}.
 \end{equation}
 In particular, by applying regularity estimates for $u$ in terms of the data $f$,
 the solution satisfies 
 \begin{equation*}
    \norm{u-u_h}_{\tilde H^s(\Omega)} \leq \left\{ \begin{array}{ll} 
                                                         C(s)\,h^{1/2}\,|\log h|\, \norm{f}_{C^{1/2-s}(\Omega)} & \text{ if } 0<s<1/2,\\
                                                         C(s)\,h^{1/2}\,|\log h|\, \norm{f}_{L^\infty(\Omega)}  & \text{ if } s=1/2,\\
                                                         \frac{C(s,\beta)}{2s-1}\,h^{1/2}\,\sqrt{\log h}\, \norm{f}_{C^\beta(\Omega)} & \text{ if } 1/2<s<1,\, \beta>0.                                                                                       \end{array}\right.
 \end{equation*}
\end{theorem}

%By using a standard Aubin-Nitsche argument, we obtain estimates in $L^2(\Omega)$; see \cite{Ern} and \cite{AiCl2018}.
%\begin{theorem}
%If the family of triangulations $\T_h$ is shape-regular and globally quasi-uniform, and $u\in H^{s+1/2-\eps}(\Omega)$ for some $\eps>0$, then
% \begin{equation*}
%    \norm{u-u_h}_{L^2(\Omega)} \leq \left\{ \begin{array}{ll} 
%                                                         C(s,\eps)\,h^{1/2+s-\eps}\, \norm{u}_{H^{s+1/2-s}(\Omega)} & \text{ if } 0<s<1/2,\\
%                                                         C(s,\eps)\,h^{1-2\eps}\, \norm{u}_{H^{s+1/2-\eps}(\Omega)}  & \text{ if } 1/2\leq s < 1.                                                         
%                                                   \end{array}\right.
% \end{equation*}   
%\end{theorem}
These results naturally assume that the stiffness matrix $A$ can be calculated exactly. 
However, looking at \eqref{eq:aij_nah} it seems that this is not possible.
Even in the simplest case (see \eqref{eq:aij_fern}) the six-dimensional integrals cannot be calculated analytically.
Therefore, we have to rely on numerical integration.
Additionally, the integrals may even become singular if the supports of the linear basis functions have a nonempty intersection.
In the next section, we present a method to lift these singularities
and we provide estimates in terms of the discretization parameter~$h$ 
indicating which quadrature rule must be used to guarantee the error estimate \eqref{eq:error_u_h}.

\section{Duffy Transformation} \label{sec:Duffy_Trafo}
The Duffy transformation is a well-known method for singularity lifting of singular integrals. 
In his original work M.~Duffy \cite{Duffy} described how to lift point singularities at one corner of a three-dimensional pyramid.
For this purpose the integration domain was transformed to the unit cube
and the singularity was lifted by the Jacobi determinant of the nonlinear transformation. 
This is done in a similar way as with spherical coordinates.

Based on this idea, Sauter and Schwab \cite{Sauter} have investigated singularities that can arise from the interaction of triangles,
such as a point singularity when the triangles touch at a corner, or a singularity along a common edge of these triangles.
Such problems occur in the context of the boundary element method for elliptic operators.

Here, their theory and the Duffy transformation are adapted to the requirements of the fractional Laplacian in three dimensions, i.e.\ to the integrals in \eqref{eq:aij_nah}:
\begin{align}
    I_{t_1,t_2} &:= \int_{t_1 } \int_{t_2} k_1(x,y)\ud y \ud x\ 
                 := \int_{t_1 } \int_{t_2} \dfrac{ [\varphi_i(x) - \varphi_i(y)] \, [\varphi_j(x) - \varphi_j(y)]} { \lvert x-y \rvert^{3+2s}}\ud y \ud x, \label{int_3d3d}\\
    I_{t,\tau}  &:= \int_t \int_\tau  k_2(x,y)\ud s_y \ud x 
                 := \int_t \varphi_i(x) \, \varphi_j(x) \int_\tau   \dfrac{ (y-x)^T n_\tau(y) } { \lvert x - y \rvert^{3+2s}}\ud s_y \ud x, \label{int_3d2d}
\end{align}
where $t,t_1,t_2 \in \mathcal{T}$, $\tau \in \mathcal{P}$ and $n_\tau$ is the normal vector with respect to the panel $\tau$.
As we can see in~\eqref{int_3d3d} and~\eqref{int_3d2d}, 
we have to study the cases of the interaction of two tetrahedra and of a tetrahedron and a triangle, respectively.
This results in a total of seven singularity cases, which must be considered individually.
However, the procedures for each case share common principles.
To understand the basic idea, we present an adapted version of the basic problem of \cite{Duffy},
\[
   I_p = \int_p f(\w) \ud \w,
\]
where $p := \{ \w \in \R^4: 0\leq \w_4 \leq \w_1, \; 0 \leq \w_3 \leq \w_2, \; 0 \leq \w_2 \leq \w_1  \text{ and } 0 \leq \w_1 \leq 1\}$,
and where $f(\w) = |\w|^{-2}$
is a function with a point singularity at $(0,0,0,0)^T$. Then we have
\begin{equation} \label{eq:advExample}
   I_p = \int_0^1 \int_0^{\w_1} \int_0^{\w_2} \int_0^{\w_1} \frac{1}{\w_1^2+\w_2^2+\w_3^2+\w_4^2}\ud \w_4 \ud \w_3 \ud \w_2 \ud \w_1.
\end{equation}   
Obviously, the singularity can be lifted with spherical coordinates. 
However, the integration domain is not well suited for this approach.
Surprisingly, the simple nonlinear transformation,
\[
   \xi := \w_1, \quad \eta_1 := \w_2/\w_1, \quad \eta_2 := \w_3/\w_2, \quad \eta_3 := \w_4/\w_1,
\]
is also sufficient to lift the singularity. 
First, we examine the Jacobi determinant $\deta \J(\xi,\eta) = \xi^3\eta_1$ of the nonlinear transformation.
Second, we consider the integration domain
\begin{align}\label{eq:domain_transform_unitcube}
  \left\{ \begin{array}{rcl}
		  0    &\leq \w_1 \leq& 1             \\
		  0    &\leq \w_2 \leq& \w_1       \\
		  0    &\leq \w_3 \leq& \w_2 \\
		  0    &\leq \w_4 \leq& \w_1             
		  \end{array} \right\} \Leftrightarrow
	  \left\{ \begin{array}{rclcl}
		  0    &\leq& \xi &\leq& 1             \\
		  0    &\leq& \xi \eta_1 &\leq& \xi       \\
		  0    &\leq& \xi \eta_1\eta_2 &\leq& \xi\, \eta_1 \\
		  0    &\leq& \xi \eta_3 &\leq& \xi         
		  \end{array} \right\}\Leftrightarrow 
	  \left\{ \begin{array}{rll}
		  0    &\leq \xi \leq& 1             \\
		  0    &\leq \eta_1 \leq& 1       \\
		  0    &\leq \eta_2 \leq& 1 \\
		  0    &\leq \eta_3 \leq& 1         
		  \end{array} \right\}  
\end{align}
and then the complete integral
\begin{align*}
   I_p &= \int_0^1 \int_0^1 \int_0^1 \int_0^1 \frac{\xi^3 \eta_1}{\xi^2 + (\xi\eta_1)^2+(\xi\eta_1\eta_2)^2+(\xi\eta_3)^2} \ud \eta_3 \ud \eta_2 \ud \eta_1 \ud \xi \\
       &= \frac12\int_0^1 \int_0^1 \int_0^1 \frac{\eta_1}{1+\eta_1^2+ (\eta_1\eta_2)^2+\eta_3^2} \ud \eta_3 \ud \eta_2 \ud \eta_1. 
\end{align*}
The decisive factor is the structure of the integration domain, which induces an ordering of the variables, i.e.\ 
\begin{equation} \label{eq:example_ordering}
    0 \leq \w_3 \leq \w_2 \leq \w_1 \leq 1 \quad \text{ and } \quad  0 \leq \w_4 \leq \w_1 \leq 1. 
\end{equation}    
This order is used by the nonlinear transformation to lift the singularity in two steps, which
is similar to using spherical coordinates, where $\xi$ takes the role of the radius. 
First, $\xi$ is factored out of the fraction, which is clearly not singular anymore, 
and second, the determinant of the Jacobi matrix is used to lift the remaining singularity with respect to $\xi$.

In order to apply this technique to our problems, we need to understand more about the ordering of the variables, which is of crucial importance in the removal of the singularity.
As it can be seen in~\eqref{eq:example_ordering}, the ordering leads to two inequality chains.
To be more precise it leads to one inequality chain, which splits up after at least one common variable (in this case $\w_1$).
We call such an inequality chain a forking inequality chain. In this situation,
after the application of the nonlinear transformation, $\xi$ can be factored out of the denominator,
which is crucial for lifting the singularity.
Therefore, we have to make sure that all variables can be described with a forking inequality chain.
To ensure such a forking chain, it may be necessary to split up the integration domain.

\subsection{Interaction between two tetrahedra} \label{sec:Duffy_3d3d}
First, we consider \eqref{int_3d3d}, i.e.\ how two tetrahedra interact with each other. 
In this case, we have to distinguish four different cases: the tetrahedra can have a common point, a common edge, a common face or can be identical.
In order to describe the resulting singularity types, the geometry will be shifted to the reference element, the unit tetrahedron
\[
   \tilde t := \{ \tx \in \R^3: 0 \leq \tx_3 \leq \tx_1 - \tx_2,\; 0 \leq \tx_2 \leq \tx_1,\; 0\leq \tx_1 \leq 1 \}.
\]
For this we need the linear mapping $\chi_t: \tilde t \to t$,
\begin{equation}\label{eq:chi_t}
    \chi_t(\tx) = M_t \tx + a_t,
\end{equation}
where $M_t := [b_t-a_t,c_t-b_t,d_t-b_t] \in \R^{3\times3}$ and $a_t,b_t,c_t$ and $d_t$ denote the vertices of the tetrahedron~$t$.  
Using the transformation theorem, we obtain for~\eqref{int_3d3d}
\begin{align} \label{eq:TT_start}
  I_{t_1,t_2} =  \int_{\tilde t } \int_{\tilde t} k_1(\chi_{t_1}(\tx), \chi_{t_2}(\ty))\, |M_{t_1}|\,|M_{t_2}|\ud \ty\ud \tx 
              =: \int_{\tilde t } \int_{\tilde t} \tilde k_1(\tx, \ty)\ud \ty\ud \tx,
\end{align}
where $|M_\star| := |\deta M_\star|$ for $\star \in\{ t_1, t_2\}$. 
Additionally, let the linear basis functions $\fie_i$ be represented in the following way,
\[
   \fie_i(x) := \alpha_{i,1:3}^T\, x + \alpha_{i,4}, \quad \alpha_i \in \R^4,
\]
where $\alpha_{i,1:3} := (\alpha_{i,1}, \alpha_{i,2}, \alpha_{i,3})^T$.
Additionally, we denote by $e_1$, $e_2$, and $e_3$ the canonical unit vectors in~$\R^3$.

\subsubsection{Point singularity}\label{sec:Duffy_3d3d_Vertex}
First, we consider the simplest case, that the two tetrahedra have a common vertex. 
Without loss of generality the mappings $\chi_{t_1}$ and $\chi_{t_2}$ are chosen in such a way that $a_{t_1} = a_{t_2}$.
This implies that the integrand $\tilde k_1$ is singular for $\tx = 0 = \ty$. 
By considering the integration domain as a set of inequalities and rewriting them,
\begin{align*}
   &\left\{ 
      0 \leq \tx_1 \leq 1,\; 
0    \leq \tx_2 \leq \tx_1, \; 
0    \leq \tx_3 \leq \tx_1 - \tx_2,\;   
0    \leq \ty_1 \leq 1,\;
0    \leq \ty_2 \leq \ty_1, \;
0    \leq \ty_3 \leq\ty_1 - \ty_2
\right\} \\ \Leftrightarrow 
  &\left\{ 
		  0         \leq \tx_1                \leq 1, \;
		  0         \leq \tx_2         \leq \tx_1,  \;
		  \tx_2     \leq \tx_2 +\tx_3 \leq \tx_1,\;
		  0         \leq \ty_1                \leq 1,\;
		  0         \leq  \ty_2        \leq \ty_1, \;
		  \ty_2     \leq \ty_2 +\ty_3 \leq \ty_1
	 \right\}, 		  
\end{align*}
we see that the variables form two disjoint inequality chains: 
\[
  0 \leq \tx_2 \leq \tx_2 + \tx_3 \leq \tx_1 \leq 1 \quad \text{ and } \quad 
  0 \leq \ty_2 \leq \ty_2 + \ty_3 \leq \ty_1 \leq 1.
\]
However, as the initial example~\eqref{eq:example_ordering} has shown, a forking inequality chain is needed to lift the singularity. 
Therefore, we connect the two chains with each other by introducing an artificial dependence on the first chain link:
\[
    \ty_1\leq \tx_1 \leq 1 \quad \text{ and } \quad   \tx_1 \leq \ty_1 \leq 1.
\]
From the geometric point of view, this means that the integration domain must be split up:
\begin{align} \label{eq:splitup_3d3dvertex}
\begin{split}
 &\left\{ 
             0     \leq \tx_1         \leq 1,\,       
             0     \leq \tx_2         \leq \tx_1,\,   
             \tx_2 \leq  \tx_2 +\tx_3 \leq \tx_1,\, 
             0     \leq \ty_1         \leq \tx_1,\, 
    	     0     \leq \ty_2         \leq \ty_1,\, 
	     \ty_2 \leq \ty_2 +\ty_3  \leq \ty_1 
 \right\} \cup \\
 &\left\{
		  0      \leq \ty_1           \leq 1,\, 
		  0      \leq \ty_2           \leq \ty_1,\, 
		  \ty_2  \leq   \ty_2 + \ty_3 \leq \ty_1,\,  
		  0      \leq \tx_1           \leq \ty_1,\, 
		  0      \leq  \tx_2          \leq \tx_1,\, 
		  \tx_2  \leq  \tx_2 + \tx_3  \leq \tx_1		  
  \right\}.
\end{split}	  
\end{align}
As a consequence, we obtain the desired forking inequality chain for each of the two domains.
For the first domain it reads
\[
  0 \leq \ty_2 \leq \ty_2+\ty_3 \leq \ty_1 \leq \tx_1 \leq 1 \quad \text{ and } \quad
  0 \leq \tx_2 \leq \tx_2 + \tx_3 \leq \tx_1 \leq 1
\]
and for the second domain
\[
  0 \leq \tx_2 \leq \tx_2 + \tx_3 \leq \tx_1 \leq \ty_1 \leq 1 \quad \text{ and } \quad
  0 \leq \ty_2 \leq \ty_2 + \ty_3 \leq \ty_1 \leq 1.
\]
Similar to the introductory example, these inequality chains are used to introduce a set of new variables $\omega \in \R^6$, $\w_1 := \tx_1$, $\w_2 := \tx_2+\tx_3$, $\w_3 := \tx_2$,
$\w_4 := \ty_1$, $\w_5 := \ty_2 + \ty_3$, and $\w_6 := \ty_2$. By interchanging the order of integration we see
for the first (and similarly for the second) domain:
	  \begin{equation*}
		\int_0^1 \int_0^{\w_1} \int_0^{\w_2} \int_0^{\w_1} \int_0^{\w_4} \int_0^{\w_5}
	            \tilde k_1 \left( \begin{bmatrix} \w_1 \\ \w_3  \\ \w_2 - \w_3  \end{bmatrix},
                                      \begin{bmatrix} \w_4 \\ \w_6  \\ \w_5 - \w_6  \end{bmatrix} \right) \ud \w.
          \end{equation*} 
Finally, each domain can be mapped onto the six-dimensional unit cube:
\begin{align*}
	          \xi    :=  \w_1,\;       \eta_1 := \w_2/\w_1,\; \eta_2 &:= \w_3/\w_2,\;
	          \eta_3 :=  \w_4/\w_1,\;  \eta_4 := \w_5/\w_4,\; \eta_5 := \w_6/\w_5,                                     
\end{align*}
where again $\deta\J(\xi,\eta)= \xi^5\,\eta_1\,\eta_3^2\,\eta_4 =: \xi^5\, \deta \J(\eta),$ denotes the Jacobi determinant of the nonlinear transformation.
The transformation of the integration domain works analogously to \eqref{eq:domain_transform_unitcube}. All in all this leads to 
\begin{align*}
   I_{t_1,t_2} &= \int_{ [0,1]} \int_{[0,1]^5} \{ \tilde k_1(\D_1(\xi,\eta)) + \tilde k_1(\D_2(\xi,\eta))\}\,  \deta \J(\xi,\eta) \ud \eta \ud \xi,               
\end{align*}
where
\begin{equation} \label{eq:simplify_D_vertex}
   \D_m(\xi,\eta):= (\D_m^1(\xi,\eta),\D_m^2(\xi,\eta)) = (\xi \D_m^1(\eta),\xi \D_m^2(\eta)) = \xi \D_m(\eta), \quad  m=1,2,
\end{equation} 
with 
\begin{align*}
   \D_1(\eta) := \left( \begin{bmatrix} 1      \\ \eta_1\eta_2         \\ \eta_1(1-\eta_2)         \end{bmatrix},
                        \begin{bmatrix} \eta_3 \\ \eta_4\eta_5\eta_6  \\ \eta_3\eta_4(1-\eta_5)  \end{bmatrix} \right)  \text{  and  }
   \D_2(\eta) := \left( \begin{bmatrix} \eta_3 \\ \eta_4\eta_5\eta_6  \\ \eta_3\eta_4(1-\eta_5)  \end{bmatrix},
                        \begin{bmatrix} 1      \\ \eta_1\eta_2         \\ \eta_1(1-\eta_2)         \end{bmatrix} \right)                 
\end{align*}
is the Duffy transformation for the $m$-th sub-domain.
To see that $\xi$ can also be factored out of
\[\tilde k_1(\tilde x,\tilde y)=\frac{[\fie_i(x)-\fie_i(y)][\fie_j(x)-\fie_j(y)]}{|x-y|^{d+2s}}|M_{t_1}||M_{t_2}|,
\]
where $x=\chi_{t_1}(\tilde x)$ and $y=\chi_{t_2}(\tilde y)$, let us first investigate the
denominator. Due to the choice of $\chi_{t_1}$ and $\chi_{t_2}$,
\begin{equation}\label{eq:linK}
  |x-y|^{-d-2s} = |\chi_{t_1}(\tx) - \chi_{t_2}(\ty)|^{-d-2s}                 
                = \xi^{-d-2s}  | M_{t_1}\D_m^1(\eta) - M_{t_2}\D_m^2(\eta)|^{-d-2s}, \quad m=1,2.  
\end{equation}

However, the examination of the product of the linear basis functions~$\fie_i$, $\fie_j$ is more complicated
due to the interaction between the linear basis functions and the tetrahedra.
%The problem here is the compact support of $\fie_i$ and $\fie_j$.
Depending on the choice of the tetrahedra, terms in the product $[\fie_i(x)-\fie_i(y)][\fie_j(x)-\fie_j(y)]$ of the differences can vanish. 
There exit three different cases: a term vanishes in both, in one, and in none of the factors.
First, we consider the case that no term vanishes.
This happens if and only if  $t_1, t_2 \subset \suppa\fie_i \cap \suppa \fie_j$. 
In this case, we obtain
\begin{equation}\label{eq:lin} 
   \fie_\star(x) - \fie_\star(y) = \alpha_{\star,1:3}^T ( M_{t_1}\tx - M_{t_2} \ty) 
                                 = \xi \alpha_{\star,1:3}^T ( M_{t_1} \D_m^1(\eta) - M_{t_2} \D_m^2(\eta))
\end{equation} 
for $m=1,2$ and for $\star = i,j$.
This implies that $\xi$ can be factored out of the difference.
For the second case, we only have to investigate the factor which contains the vanishing term, because
the other factor can be treated in the same way as \eqref{eq:lin}.
Without loss of generality, we assume that $t_2\not\subset \suppa \fie_i$ and $t_1,t_2\subset \suppa \fie_j$.
Moreover, we know that $a_{t_1}=a_{t_2} \in t_2$, which implies that $\fie_i(a_{t_1}) = 0$ and that
\begin{equation}\label{eq:lin2}
   \fie_i(x) = \fie_i(\chi_{t_1}(\tx)) = \alpha_{i,1:3}^T M_{t_1}\tx + \fie_i(a_{t_1}) = \alpha_{i,1:3}^T M_{t_1}\tx = \xi \alpha_{i,1:3}^T M_{t_1} \D_m^1(\eta)
\end{equation}
for $m=1,2.$
The last case happens if and only if $\suppa \fie_i \cap \suppa \fie_j = a_{t_1}$.
Then the ideas used in~\eqref{eq:lin2} can be applied twice.

By combining the results of~\eqref{eq:linK}, \eqref{eq:lin}, and~\eqref{eq:lin2}, it is shown that $\xi$ can be factored out of $\tilde k_1$,
\begin{align} \label{eq:Ergebnis_Duffy_3d3d_Vertex}
\begin{split} 
   I_{t_1,t_2} 
    = \frac{1}{5-2s} \,  \int_{[0,1]^5} \tilde k_{1,V}(\eta) \ud \eta,\quad
    \tilde k_{1,V}(\eta):= [ \tilde k_1(\D_1(\eta)) + \tilde k_1(\D_2(\eta))]\,  \deta\J(\eta).
 \end{split}   
\end{align}
Since the integral relative to $\xi$ can be integrated analytically, only a five-dimensional integral has to be integrated numerically.
The regularity of the remaining integrand will be proved in Lemma~\ref{lem:analyt_3d3d_Tetra}.

\subsubsection{Singularity along an edge}\label{sec:duffy3d3dedge}
For the next case, we assume that the tetrahedra $t_1$ and $t_2$ have a common edge.
Without loss of generality the mappings $\chi_{t_1}$ and $\chi_{t_2}$ are chosen in such a way that $a_{t_1} = a_{t_2}$ and $b_{t_1} = b_{t_2}$.
This leads to $\chi_{t_1}(\kappa\, e_1) = \chi_{t_2}(\kappa\, e_1)$ for all $\kappa \in [0,1]$.
Therefore, $\tilde k_1$ is singular, if and only if $\tx_1 = \ty_1$ and if the remaining variables are zero.
We introduce local coordinates to describe the singularity along the edge as a five-dimensional point singularity.
For the first domain choose $\tz \in \R^5$ as $\tz_1 := \ty_1 - \tx_1$, $\tz_i := \ty_i$, $\tz_{i+2} := \tx_i$,  $i = 2,3$,
and for the second domain choose $\tz \in \R^5$ as $\tz_1 := \tx_1 - \ty_1$, $\tz_i := \tx_i$, $\tz_{i+2} := \ty_i$, $i = 2,3$.
Using the same splitting of the domain as in the previous case, \eqref{eq:splitup_3d3dvertex} becomes 
\begin{align*}
\begin{split}
 &\left\{ 
             0      \leq \tx_1  \leq 1,
             0      \leq \tz_4  \leq \tx_1,
             0      \leq \tz_5  \leq \tx_1 - \tz_4,
             -\tx_1 \leq \tz_1  \leq 0,
    	     0      \leq \tz_2  \leq \tx_1 + \tz_1,
	     0      \leq \tz_3  \leq \tx_1 + \tz_1 - \tz_2 
         \right\} \cup \\
 &\left\{ 
		  0     \leq \ty_1  \leq 1,
		  0     \leq \tz_4  \leq \ty_1,   
		  0     \leq \tz_5  \leq \ty_1 - \tz_4,  
		 -\ty_1 \leq \tz_1  \leq 0,
		  0     \leq \tz_2  \leq \ty_1 + \tz_1,
		  0     \leq \tz_3  \leq \ty_1 + \tz_1 - \tz_2		  
 \right\}.
\end{split}	  
\end{align*}
The integration order does not have to be changed, 
because the singularity is already at the boundary of each domain.
The next step is to introduce an ordered set of variables.
The first sub-domain yields the following forking inequality chains 
\[ 
   0 \leq -\tz_1 \leq -\tz_1 + \tz_2 \leq -\tz_1 + \tz_2 + \tz_3 \leq \tx_1 \leq 1 \quad \text{and} \quad 0 \leq \tz_4 \leq \tz_4 + \tz_5 \leq \tx_1 \leq 1.
\]
The forking inequality chain splits after the first chain link $\tx_1$. However, the set for which $\tilde k_1$ is singular is characterized by~$\tilde z=0$.
Therefore, we have to divide the sub-domain again to push the splitting one link further.
This is done by applying the same ideas as in~\eqref{eq:splitup_3d3dvertex} to obtain
\[
   -\tz_1 + \tz_2 \leq -\tz_1 + \tz_2 + \tz_3 \leq \tz_4 + \tz_5 \quad \text{and} \quad \tz_4 \leq \tz_4 + \tz_5 \leq -\tz_1 + \tz_2 + \tz_3.
\]
 Using the symmetry between the two sub-domains, the Duffy transformation can be simplified to
 \[
  I_{t_1,t_2} = \int_{[0,1]^6} \sum_{m=1}^4 \tilde k_1(\D_m(\xi,\eta))\, \deta\J_m(\xi,\eta) \ud \xi \ud \eta,     
 \]
 where $\D_m$ are again the Duffy transformation of the $m$-th sub-domain and $\D_3 = \D_1^T$, $\D_4 = \D_2^T$ due to symmetry.
 If we take a closer look at the $\D_m$, one can notice these transformations have a structure which can be used to simplify the integrals
 \begin{align} \label{eq:simplify_D_edge} 
 \begin{split}
   \D_m(\eta,\xi) &= (\D_m^1(\xi,\eta), \D_m^2(\xi,\eta)) = ( \xi_1\, e_1 + \xi_1\,\xi_2\, \D_m^1(\eta), \xi_1\, e_1 + \xi_1\,\xi_2\, \D_m^1(\eta))  \\
                  &= \xi_1\, (e_1,e_1) + \xi_1\,\xi_2\, \D_m(\eta),
 \end{split}  
 \end{align}
 where 
 \[
   \D_1(\eta) = \left(\begin{bmatrix}  0                      \\ \eta_1                      \\ 1 - \eta_1           \end{bmatrix},
		      \begin{bmatrix}  -\eta_2\eta_3\eta_4 \\ \eta_2\eta_3\,(1  - \eta_4) \\ \eta_2( 1 - \eta_3)  \end{bmatrix} \right), \quad
   \D_2(\eta) = \left(\begin{bmatrix}  0              \\ \eta_3\eta_4      \\ \eta_3( 1-  \eta_4)    \end{bmatrix},
		      \begin{bmatrix}  -\eta_1\eta_2 \\ \eta_1( 1 -\eta_2)  \\  1 - \eta_1             \end{bmatrix}    \right).	  
 \]
 Notice that we can almost use the same arguments as in~\eqref{eq:Ergebnis_Duffy_3d3d_Vertex}
 to factor $\xi_1$ and $\xi_2$ out of $\tilde k_1$. Additionally, we use that $\fie_i|_{[a_{t_1},b_{t_1}]} = 0$ due to $\chi_{t_1}(\xi_1\,e_1) \subset [a_{t_1},b_{t_1}]$. 
 Thus, we obtain 
 \begin{align}\label{eq:Ergebnis_Duffy_3d3d_Edge}
 \begin{split}
  I_{t_1,t_2}
       = \frac{1}{5-2s}\,\frac{1}{4-2s}\, \int_{[0,1]^4} \tilde k_{1,E}(\eta) \ud\eta,\quad
       \tilde k_{1,E}(\eta):=\sum_{m=1}^4\, \tilde k_1(\D_m(\eta))\, \deta\J_m(\eta).
 \end{split}     
 \end{align}
 As one can see, after the Duffy transformation only a four-dimensional integral has to integrated numerically instead of a six-dimensional one.
 The regularity of the remaining integrand will be shown in Lemma~\ref{lem:analyt_3d3d_Tetra}.

 \subsubsection{Singularity along a face} \label{sec:duffy3d3dface}
 Let $t_1$ and $t_2$ have a common face $\tau$.
 Without loss of generality the mappings $\chi_{t_1}$ and $\chi_{t_2}$ are chosen in such a way that $a_{t_1} = a_{t_2}$, $b_{t_1} = b_{t_2}$, and $c_{t_1} = c_{t_2}$.
 This leads to 
 \[
    \chi_{t_1}(\kappa_1\,e_1+\kappa_2\,e_2) = \chi_{t_2}(\kappa_1\,e_1+\kappa_2\,e_2)  \; \text{ for all } \kappa_1,\kappa_2 \in [0,1].
 \]
 Therefore, $\tilde k_1$ is singular if $\tx_i = \ty_i$ for $i=1,2$ and $\tx_3 = 0 = \ty_3$.
 First, we introduce local variables $\tz \in \R^4$ to describe the singularity as a four-dimensional point singularity,
 $\tz_i := \ty_i - \tx_i$, $i=1,2$, $\tz_3 := \ty_3$ and $\tz_4 := \tx_3$.
  Hence, we obtain by starting with~\eqref{eq:TT_start}
\begin{equation*}
  I_{t_1,t_2} = \int_0^1 \int_0^{\tx_1}  
  \int_{-\tx_1}^{1-\tx_1} \int_{-\tx_2}^{\tz_1 + \tx_1 - \tx_2} \int_{0}^{\tz_1-\tz_2 + \tx_1 - \tx_2} \int_0^{\tx_1 - \tx_2}
  \tk_1\left(   \begin{bmatrix} \tx_1 \\ \tx_2 \\ \tz_4 \end{bmatrix},
            \begin{bmatrix} \tz_1 + \tx_1 \\ \tz_2 + \tx_2  \\ \tz_3 \end{bmatrix}\right) \ud  \tz\ud  \tx_{1:2}.
\end{equation*}
 Since we want to apply the ideas from the example at the beginning of Section~\ref{sec:Duffy_Trafo}, 
 the singularity has to be shifted to the boundary of the integration domain by splitting it up.
 To simplify this procedure, the order of integration w.r.t.\ $\tx_{1:2}$ and $\tz_{1:3}$ is reversed, which results in
\[
  \left\{ \begin{array} {rcl}
   -1                                                                       &\leq \tz_1  \leq& 1 \\
   \max\{-1,\,-1+\tz_1\}                                                    &\leq \tz_2  \leq& \text{min}\{1,\,1+\tz_1\} \\
    0                                                                       &\leq \tz_3  \leq& \text{min}\{1,\,1+\tz_1\} + \text{min}\{0,\, -\tz_2\} \\
    \text{max} \{ 0,\, -\tz_2,\,-\tz_1 + \tz_3,\,  \tz_2 - \tz_1 + \tz_3\}  &\leq \tx_1  \leq& \text{min} \{ 1,\, 1 - \tz_1\}\\
    \text{max} \{ 0,\, -\tz_2\}                                             &\leq \tx_2  \leq& \tx_1 + \text{min} \{0 ,\, \tz_1 - \tz_2 - \tz_3\} \\
    0                                                                       &\leq \tz_4  \leq& \tx_1 - \tx_2
   \end{array} \right\}.
\]
 The following procedure is similar to the two singularity cases before. 
The domain is subdivided to push the point singularity to the boundary. 
In total, we split the set into nine sub-domains to resolve the min and max constraints. 
Furthermore, eight of these domains have to be refined again to satisfy a suitable forking inequality chain in the same way as in Section \ref{sec:duffy3d3dedge}. 
Here, only the final result after the non-linear transformation to the six-dimensional unit cube is presented.
In the same sense as in~\eqref{eq:simplify_D_vertex} and~\eqref{eq:simplify_D_edge}
the non-linear transformations~$\D_m$ for the domains share a common structure which will be used to simplify them:
 \begin{equation}\label{eq:simplify_D_face}
 \begin{split}
    \D_m(\xi,\eta) &= ( \xi_1\,e_1 + \xi_1[1-\xi_2]\,e_2 + \xi_1\xi_2\xi_3\,\D_m^1(\eta), 
                        \xi_1\,e_1 + \xi_1[1-\xi_2]\,e_2 + \xi_1\xi_2\xi_3\,\D_m^2(\eta)) \\
                   &= \xi_1\,(e_1,e_1) + \xi_1[1-\xi_2]\,(e_2,e_2) + \xi_1\xi_2\xi_3\,\D_m(\eta)
 \end{split}   
 \end{equation}
 with
 \begin{alignat*}{2}
  \D_1(\eta) &= \left( \begin{bmatrix}  0                      \\ \eta_1                 \\ 1 - \eta_1                     \end{bmatrix},
                       \begin{bmatrix}  -\eta_1\eta_2\eta_3 \\ 0                      \\ \eta_1\eta_2 ( 1 - \eta_3) \end{bmatrix} \right), \quad 
  &&\D_2(\eta)  = \left( \begin{bmatrix}  0                      \\ \eta_1\eta_2           \\ \eta_1 (1 - \eta_2)          \end{bmatrix},
                       \begin{bmatrix}  -\eta_1\eta_2\eta_3 \\ 0                      \\  1 - \eta_1\eta_2\eta_3    \end{bmatrix} \right), \\                       
  \D_3(\eta) &= \left( \begin{bmatrix}  0                      \\ \eta_1\eta_2         \\ 1 - \eta_1\eta_2             \end{bmatrix},
                       \begin{bmatrix}  -\eta_1\eta_2\eta_3 \\ 0                      \\ \eta_1 ( 1 - \eta_2\,\eta_3) \end{bmatrix} \right), \quad
  &&\D_4(\eta)  = \left( \begin{bmatrix}  0                      \\ \eta_1\eta_2\eta_3 \\ \eta_1\eta_2 (1 - \eta_3)  \end{bmatrix},
                       \begin{bmatrix}  -\eta_1                 \\ 0                      \\  1 - \eta_1                    \end{bmatrix} \right), \\ 
  \D_5(\eta) &= \left( \begin{bmatrix}  0                      \\ \eta_1\eta_2\eta_3 \\ \eta_1 (1 - \eta_2\,\eta_3)  \end{bmatrix},
                       \begin{bmatrix}  -\eta_1\eta_2         \\ 0                      \\  1 - \eta_1\,\eta_2            \end{bmatrix} \right), \quad
  &&\D_6(\eta)  = \left( \begin{bmatrix}  0                      \\ \eta_1\eta_2\eta_3 \\ 1 - \eta_1\eta_2\eta_3    \end{bmatrix},
                       \begin{bmatrix}  -\eta_1\eta_2         \\ 0                      \\ \eta_1(1 - \eta_2)          \end{bmatrix} \right), \\ 
  \D_7(\eta) &= \left( \begin{bmatrix}  0                      \\ 0                      \\ 1                              \end{bmatrix},
                       \begin{bmatrix}  -\eta_1\eta_2\eta_3 \\ \eta_1\eta_2(1- \eta_3)      \\ \eta_1(1 - \eta_2)          \end{bmatrix} \right), \quad
  &&\D_{15}(\eta) = \left( \begin{bmatrix} -\eta_1\eta_2\eta_3 \\ \eta_1\eta_2( 1 - \eta_3)  \\ 1 - \eta_1\eta_2   \end{bmatrix},
                       \begin{bmatrix}  0                         \\ 0                            \\ \eta_1                 \end{bmatrix} \right).
 \end{alignat*}
 The transformations $\D_8$ to $\D_{14}$ are the symmetric versions of $\D_1$ to $\D_7$, i.e.\ $\D_{7+i}(\eta) = \D_i^T(\eta)$ for $i=1,\ldots,7$.
The remaining two are described by
\begin{align*}
  \D_{16}(\eta) &= \left( \begin{bmatrix}  0                \\ 0                 \\ \eta_3      \end{bmatrix},
                          \begin{bmatrix}  -\eta_1\eta_2   \\ \eta_1(1-\eta_2)  \\ 1 - \eta_1  \end{bmatrix} \right) \quad \text{and} \quad
  \D_{17}(\eta)  = \left( \begin{bmatrix} -\eta_1\eta_2   \\ \eta_1(1-\eta_2)  \\ 1 - \eta_1  \end{bmatrix},
                          \begin{bmatrix}  0                \\ 0                 \\ \eta_1\eta_3      \end{bmatrix} \right).
\end{align*}  
For the corresponding Jacobi determinants, it holds that $\deta\J_m(\xi,\eta) = \deta\J(\xi)\, \deta\J_m(\eta)$ for $m=1,\ldots,17$,
where $\deta\J(\xi) = \xi_1^5\xi_2^4\xi_3^3$ and $\deta\J_m(\eta)=\eta_1^2\eta_2,$   $m=1, \ldots,15$.
The remaining two cases are $\deta\J_{16}(\eta) = \eta_1$ and $ \deta\J_{17}(\eta) =\eta_1^2.$
All in all, the Duffy transformation can now be applied
\begin{align} \label{eq:Ergebnis_Duffy_3d3d_Face}
\begin{split}
   I_{t_1,t_2}      
      = \prod_{j=0}^2 \frac{1}{5-j-2s}  \int_{[0,1]^3} \tilde k_{1,F}(\eta) \ud\eta,\quad
      \tilde k_{1,F}(\eta):=\sum_{m=1}^{17}\, \tilde k_1(\D_m(\eta))\, \deta\J_m(\eta).
\end{split}      
\end{align}
Again, we used the linearity of the basis functions and the transformation between the reference tetrahedron and the original tetrahedron to simplify
the Duffy transformation. Since the procedure is completely analogous to the case of the singularity along an edge, we omit the details. 
The only difference is to consider that in the case of a single vanishing term,
it holds that $\chi_{t_1}(\kappa_1\,e_1+\kappa_2\,e_2) \in \tau$ and $\fie_i|_\tau = 0$.
As one can see, after the Duffy transformation only a three-dimensional integral has to be integrated numerically instead of a six-dimensional one.
 The regularity of the remaining integrand will be shown in Lemma~\ref{lem:analyt_3d3d_Tetra}.

\subsubsection{Two identical tetrahedra}\label{sec:Duffy_3d3d_Tetra}
We now deal with the strongest singularity, i.e.\ the two tetrahedra coincide. 
Since $t_1$ equals $t_2$, we drop the index and $\tilde k_1$ is singular if and only if  $\tx = \ty$.
Starting from~\eqref{eq:TT_start} we introduce local variables $\tz \in \R^3, \tz := \ty - \tx$, to describe the singularity as a three-dimensional point singularity, which leads to
\begin{equation*}
  I_{t,t} = \int_0^1 \int_0^{\tx_1} \int_0^{\tx_1 - \tx_2} 
  \int_{-\tx_1}^{1-\tx_1} \int_{-\tx_2}^{\tz_1 + \tx_1 - \tx_2} \int_{-\tx_3}^{\tz_1-\tz_2 +\tx_1- \tx_2-\tx_3} 
  \tilde k_1(  \tx, \tz + \tx) \ud  \tz\ud  \tx.
\end{equation*}
Then the point singularity is moved to the boundary by swapping the integration order and splitting up the integration domain:
\[ \left\{ \begin{array} {rcl}
   -1                                                                           &\leq \tz_1  \leq& 1 \\
   -1 + \text{max}\{0,\,\tz_1\}                                                 &\leq \tz_2  \leq& 1 + \text{min}\{0,\,\tz_1\}\\
   -1 + \text{max}\{0,\,\tz_1\} + \text{max}\{0,\,-\tz_2\}                      &\leq \tz_3  \leq& 1 + \text{min}\{0,\,\tz_1\} + \text{min}\{0,\,-\tz_2\}\\
    \text{max} \{ 0,\, -\tz_1,\, - \tz_2,\, -\tz_3, \quad                       &                & \\ 
    \, -\tz_2-\tz_3,\, \tz_2 -\tz_1,\, -\tz_1 + \tz_3,\, \tz_2 -\tz_1 + \tz_3\} &\leq \tx_1  \leq& 1 + \text{min}\{0 ,\, -\tz_1\} \\
    \text{max} \{ 0,\, -\tz_2\}                                                 &\leq \tx_2  \leq& \tx_1 + \text{min} \{0 ,\, \tz_1 - \tz_2, \, \tz_1 - \tz_2 -\tz_3\} \\
    \text{max} \{ 0,\, -\tz_3\}                                                 &\leq \tx_3  \leq& \tx_1 - \tx_2 + \text{min} \{ 0,\, \tz_1 -\tz_2 - \tz_3\}
   \end{array} \right\}. \] 
 The split helps to resolve the min and max conditions.  
 All in all, we obtain $18$ integration domains each of which satisfies a forking inequality chain.
 This ultimately leads to:
\begin{align} \label{eq:Ergebnis_Duffy_3d3d_Tetra}
\begin{split}
   I_{t,t} &= \int_{[0,1]^4} \int_{[0,1]^2}  \sum_{m=1}^{18} \tilde k_1(\D_m(\xi,\eta))\, \deta\J_m(\xi,\eta) \ud \eta \ud \xi \\
           & =2\prod_{l=0}^3 \frac{1}{5-l-2s} \int_{[0,1]^2} \tilde k_{1,T}(\eta) \ud\eta,\quad
           \tilde k_{1,T}(\eta):=\sum_{m=1}^{9} \tilde k_1(\D_m(\eta))\, \deta\J_m(\eta),
\end{split}      
\end{align}
with $\D_m(\xi,\eta) = \xi_1\xi_2\xi_3\xi_4\, \D_m(\eta)$ for $m =1\ldots,18$,
where
\begin{align*}
\D_1(\eta)   &=  \left(\begin{bmatrix} 0        \\ \eta_1             \\  - \eta_1                       \end{bmatrix},
                         \begin{bmatrix} -\eta_1\eta_2 \\ 0                    \\  - 1                       \end{bmatrix}\right), 
\D_2(\eta)    =  \left(\begin{bmatrix} 0        \\ 1             \\ 0                       \end{bmatrix},
                         \begin{bmatrix} -\eta_1\eta_2 \\ 0                    \\ \eta_1(1 - \eta_2) \end{bmatrix}\right), 
\D_3(\eta)   =  \left(\begin{bmatrix} 0        \\ \eta_1             \\ - 1                       \end{bmatrix},
                         \begin{bmatrix} -\eta_1\eta_2 \\ 0                    \\ - \eta_1\eta_2                       \end{bmatrix}\right), \\ 
\D_4(\eta)    &=  \left(\begin{bmatrix} 0        \\ \eta_1\eta_2             \\ - \eta_1\eta_2                       \end{bmatrix},
                         \begin{bmatrix} -\eta_1 \\ 0                    \\ - 1                       \end{bmatrix}\right),  
\D_5(\eta)   =  \left(\begin{bmatrix} 0        \\ \eta_1\eta_2             \\ \eta_1(1 - \eta_2) \end{bmatrix},
                         \begin{bmatrix} -1        \\ 0                    \\   0                  \end{bmatrix}\right),
\D_6(\eta)    =  \left(\begin{bmatrix} 0        \\ \eta_1\eta_2             \\ - 1                       \end{bmatrix},
                         \begin{bmatrix} -\eta_1 \\ 0                    \\ - \eta_1                       \end{bmatrix}\right), \\     
\D_7(\eta)   &=  \left(\begin{bmatrix} 0        \\ 0                    \\ 0                                  \end{bmatrix},
                         \begin{bmatrix} -\eta_1\eta_2 \\ \eta_1(1 - \eta_2)  \\ - 1                       \end{bmatrix}\right),
\D_8(\eta)    =  \left(\begin{bmatrix} 0        \\ 0                    \\ - 1                       \end{bmatrix},
                         \begin{bmatrix} -\eta_1\eta_2 \\ \eta_1(1 - \eta_2)  \\ - \eta_1                       \end{bmatrix} \right), 
\D_9(\eta)   =  \left(\begin{bmatrix} 0        \\ 0                    \\ \eta_1            \end{bmatrix},
                         \begin{bmatrix} -\eta_2 \\ 1 - \eta_2         \\ 0                       \end{bmatrix}\right).                       
\end{align*}
The factor $2$ in \eqref{eq:Ergebnis_Duffy_3d3d_Tetra} is due to the symmetry. 
Since $\D_{m+9}(\eta) = \D_m(\eta)^T$ for $m=1,\ldots,9$ and since $t_1$ equals $t_2$, the corresponding mappings are identical, too.
Moreover, the differences of the linear basis functions do not vanish. 
Therefore, the simplifications of the $\D_m$ and the first step in~\eqref{eq:Ergebnis_Duffy_3d3d_Tetra} are easily seen.
As one can see, after the Duffy transformation only a two-dimensional integral has to be integrated numerically instead of a six-dimensional one.
 The regularity of the remaining integrand will be shown in Lemma~\ref{lem:analyt_3d3d_Tetra}.

\subsection{Interaction between a tetrahedron and a panel}\label{sec:Duffy_3d2d}
The procedure here is similar to the procedure in Section~\ref{sec:Duffy_3d3d}. 
The main difference is that the interaction between a tetrahedron $t$ and a panel $\tau$ has to be examined.
There are three different cases: a common point, a common edge or the panel $\tau$ is a face of the tetrahedron $t$.
Again, to describe the resulting singularity types, $t$ and $\tau$ will be mapped to the corresponding reference element.
For the panel this the unit panel $\tilde \tau:= \{ \tilde y \in \R^2: 0 \leq \ty_2 \leq \ty_1,\, 0\leq \ty_1\leq 1 \}$.
The transformation will be done by the mapping $\chi_\tau: \tilde \tau \to \tau$,
\begin{equation}\label{eq:chi_tau}
  \chi_\tau(\ty) = M_\tau \ty + a_\tau,
\end{equation}
where $M_\tau := [b_\tau - a_\tau, c_\tau-b_\tau] \in \R^{3\times 2}$ and $a_\tau,b_\tau$, and $c_\tau$ denote the vertices of the panel $\tau$.
Then \eqref{int_3d2d} reads
\begin{equation} \label{eq:TP_start}
 I_{t,\tau} = \int_{\tilde t} \int_{\tilde \tau} \tilde k_2(\tx,\ty) \ud \ty \ud \tx, \quad
 \tilde k_2(\tx,\ty):=k_2(\chi_t(\tx),\chi_\tau(\ty))\, |\deta M_t|\, |(b_\tau - a_\tau) \times (c_\tau-b_\tau)|.
\end{equation}
In the following, the canonical unit vectors of $\R^2$ are denoted by $\hat e_1$ and $\hat e_2$.

\subsubsection{Point singularity} \label{sec:Duffy_3d2d_Vertex}
We start again with the simplest case that the tetrahedron and the triangle have a common vertex. 
Without loss of generality the linear mappings $\chi_t $ and $\chi_\tau$ are chosen in such a way that $a_t = a_\tau$.
 Since the procedure is analogous to the procedure in Section~\ref{sec:Duffy_3d3d_Vertex}, only the result of the Duffy transformation is presented.
 We obtain two sub-domains and for the integrals this implies that
 \[
   I_{t,\tau} = \int_{[0,1]} \int_{[0,1]^4} \sum_{m=1}^2 \tilde k_2( \D_m(\xi,\eta))\, \deta\J_m(\xi,\eta) \ud \eta \ud \xi,
 \]
 where again $\D_m(\xi,\eta) = \xi\,(\D_m^1(\eta),\D_m^2(\eta))=\xi\,\D_m(\eta)$, $m=1,2$, is the Duffy transformation for the $m$-th sub-domain with
 \begin{align*}
   \D_1(\eta) = \left( \begin{bmatrix} 1 \\ \eta_1\eta_2 \\ \eta_1(1-\eta_2) \end{bmatrix},
                       \begin{bmatrix} \eta_3 \\ \eta_3\eta_4 \end{bmatrix} \right)\; \text{ and } \;
   \D_2(\eta) = \left( \begin{bmatrix} \eta_2 \\ \eta_2\eta_3\eta_4 \\ \eta_2\eta_3(1-\eta_4) \end{bmatrix},
                       \begin{bmatrix} 1 \\ \eta_1 \end{bmatrix} \right),                     
 \end{align*}
 and where $\deta \J_m(\xi,\eta) = \xi^4\, \deta \J_m(\eta)$ with  $\deta \J_1(\eta) = \eta_1\eta_3$ and $\deta\J_2(\eta) = \eta_2^2\eta_3$.
 This can be seen similarly to \eqref{eq:simplify_D_vertex}.
 In contrast to Section~\ref{sec:Duffy_3d3d} we now have consider a different integrand~$\tilde k_2$.
 It is actually easier to examine $\tilde k_2$ because we do not have to distinguish several cases.
 First, consider the difference of the integration variables
 \begin{align*}
   y-x = \chi_\tau(\ty) - \chi_t(\tx) = M_\tau \ty - M_t \tx = \xi\, ( M_\tau \D_m^2(\eta) - M_t \D_m^1(\eta) ).
 \end{align*}
 For the inner integrand of~\eqref{eq:TP_start} this means
 \begin{align*}
  \frac{(y-x)^Tn_\tau}{|y-x|^{3+2s}} = \xi^{-2-2s} \, \frac{(M_\tau \D_m^2(\eta) - M_t \D_m^1(\eta))^T n_\tau}{|M_\tau \D_m^2(\eta) - M_t \D_m^1(\eta)|^{3+2s}},\quad m=1,2.
 \end{align*}
 The second part is analogous to the approach in Section~\ref{sec:Duffy_3d3d}. 
 Since $\tau \subset \partial(\supp\,\fie_i \cup \supp\,\fie_j)$, it follows without loss of generality that $\fie_i(a_t) = 0$ and
 \[
  \fie_i(x) = \alpha_{i,1:3}^T\, M_t \tx + \alpha_{i,4} a_t =  \alpha_{i,1:3}^T\, M_t \D_m^2(\xi,\eta) + \fie_i(a_t) = \xi\,\alpha_{i,1:3}^T\, M_t \D_m^2(\eta), \quad m=1,2.
 \]
 By combining these results, we have shown that $\xi$ can be factored out of $\tilde k_2$:
 \begin{equation}\label{eq:Ergebnis_Duffy_3d2d_Vertex}
    I_{t,\tau} = \frac{1}{4-2s}\, \int_{[0,1]^4} \tilde k_{2,V}(\eta) \ud \eta,\quad
    \tilde k_{2,V}(\eta):=\sum_{m=1}^2\, \tilde k_2(\D_m(\eta))\, \deta\J_m(\eta).
\end{equation} 
 Since the integral w.r.t.\ $\xi$ can be integrated analytically, only a four-dimensional integral has to be integrated numerically.
The regularity of the remaining integrand will be proved in Lemma~\ref{lem:analyt_3d2d_Face}.

\subsubsection{Singularity along an edge}\label{sec:Duffy_3d2d_Edge}
We consider the case that $t$ and $\tau$ share a common edge.
Adapting the approach of Section~\ref{sec:duffy3d3dedge}, we obtain four sub-domains:
\[
 I_{t,\tau} = \int_{[0,1]^5} \sum_{m=1}^4 \tilde k_2(\D_m(\xi,\eta))\, \deta\J_m(\xi,\eta) \ud \xi \ud \eta    
\]
with $\D_m$ being the Duffy-transformations on the corresponding domains
and $\deta \J_m$ denoting the Jacobi determinant of the $m$-th sub-domain,
$\deta\J_m(\xi,\eta) = \xi_1^4\xi_2^3\eta_2 =: \deta\J(\xi)\, \deta\J_m(\eta)$, $m=1,2,3,$
and 
 $\deta\J_4(\xi,\eta) =  \xi_1^4\xi_2^3\eta_1^2\,\eta_2 =: \deta\J(\xi)\, \deta\J_4(\eta)$.
Using the same idea as in~\eqref{eq:simplify_D_edge}, we obtain
\[
   \D_m(\eta,\xi) = (\D_m^1(\xi,\eta), \D_m^2(\xi,\eta)) = ( \xi_1\, e_1 + \xi_1\xi_2\, \D_m^1(\eta), \xi_1\, \hat e_1 + \xi_1\xi_2\, \D_m^1(\eta)) 
                  =: \xi_1\, (e_1,\hat e_1) + \xi_1\xi_2\, \D_m(\eta)
 \]
with
\begin{alignat*}{2}
   \D_1(\eta) &= \left( \begin{bmatrix} 0 \\ \eta_1     \\ 1 - \eta_1 \end{bmatrix},
                       \begin{bmatrix} -\eta_2\eta_3  \\ \eta_2(1 - \eta_3) \end{bmatrix} \right), \quad
   &&\D_2(\eta) = \left( \begin{bmatrix} 0 \\ \eta_2\eta_3  \\ \eta_2(1-\eta_3)  \end{bmatrix},
		       \begin{bmatrix} -\eta_1  \\ 1 - \eta_1       \end{bmatrix} \right), \\
   \D_3(\eta) &= \left( \begin{bmatrix} -\eta_2\eta_3  \\ \eta_2(1 - \eta_3)  \\ 1 - \eta_2  \end{bmatrix},
     		       \begin{bmatrix} 0                                       \\ \eta_1 \end{bmatrix} \right), \quad
   &&\D_4(\eta) = \left( \begin{bmatrix} -\eta_1\eta_2\eta_3 \\ \eta_1\eta_2(1- \eta_3) \\ \eta_1(1- \eta_2) \end{bmatrix},
                       \begin{bmatrix} 0  \\ 1               \end{bmatrix} \right).     
\end{alignat*}
Applying the tricks presented in the previous sections, $\xi_1$ and $\xi_2$ can be factored out of $\tilde k_2$.
This results in  
\begin{equation}\label{eq:Ergebnis_Duffy_3d2d_Edge}
 I_{t,\tau} =  \frac{1}{5-2s}\,\frac{1}{4-2s}\,\int_{[0,1]^3}  \tilde k_{2,E}(\eta) \ud \eta,\quad
 \tilde k_{2,E}(\eta):= \sum_{m=1}^4  \tilde k_2(\D_m(\eta))\, \deta\J_m(\eta).
\end{equation}
Hence, only a three-dimensional integral has to be integrated numerically.
The regularity of the remaining integrand will be proved in Lemma~\ref{lem:analyt_3d2d_Face}.

\subsubsection{Singularity along a face}\label{sec:Duffy_3d2d_Face}
The last case to consider is that the panel is a face of the tetrahedron.
Without loss of generality we assume that $a_t = a_\tau$, $b_t=b_\tau$, and $c_t = c_\tau$. 
This means that $\chi_t(\kappa_1\,e_1+\kappa_2\,e_2) = \chi_\tau(\kappa_1\,\hat e_1+\kappa_2\, \hat e_2)$ for all $\kappa_1,\kappa_2 \in [0,1]^2$.
In order to lift the singularity we obtain nine domains:
\begin{align*}
 I_{t,\tau}=  \int_{[0,1]^5} \sum_{m=1}^9 \tilde k_2(\D_m(\xi,\eta))\, \deta\J_m(\xi,\eta) \ud \eta \ud \xi,   
\end{align*}
where we can apply the idea from~\eqref{eq:simplify_D_face} to
 \begin{align*}
    \D_m(\xi,\eta) &= ( \xi_1\,e_1 + \xi_1(1-\xi_2)\,e_2 + \xi_1\xi_2\xi_3\, \D_m^1(\eta), 
                        \xi_1\,\hat e_1 + \xi_1(1-\xi_2)\, \hat e_2 + \xi_1\xi_2\xi_3\, \D_m^2(\eta)) \\
                   &= \xi_1\,(e_1,\hat e_1) + \xi_1(1-\xi_2)\,(e_2,\hat e_2) + \xi_1\xi_2\xi_3\, \D_m(\eta) 
 \end{align*}
with 
\begin{alignat*}{3}
  \D_1(\eta) &= \left(\begin{bmatrix}  0 \\  \eta_1   \\ 1 - \eta_1  \end{bmatrix},
		       \begin{bmatrix}  -\eta_1\eta_2 \\ 0                                        \end{bmatrix} \right), 
  &&\D_2(\eta)  = \left( \begin{bmatrix}  0 \\ \eta_1\eta_2   \\ \eta_1 (1 - \eta_2)  \end{bmatrix},
		       \begin{bmatrix}  -1  \\ 0                                        \end{bmatrix} \right), 
  &&\D_3(\eta) = \left( \begin{bmatrix}  0  \\ \eta_1\eta_2   \\ 1 - \eta_1\eta_2   \end{bmatrix},
		       \begin{bmatrix}  -\eta_1  \\ 0                                \end{bmatrix} \right), \\
  \D_4(\eta)  &= \left( \begin{bmatrix}  0  \\ 0         \\ 1 \end{bmatrix},
		       \begin{bmatrix}  -\eta_1\eta_2  \\ \eta_1(1 - \eta_2) \end{bmatrix} \right), 
  &&\D_5(\eta) = \left( \begin{bmatrix}  -\eta_1\eta_2  \\ \eta_1(1 - \eta_2)         \\ 1 -\eta_1 \end{bmatrix},
		       \begin{bmatrix}  0  \\ 0 \end{bmatrix} \right),  
  &&\D_6(\eta)  = \left( \begin{bmatrix}  -\eta_1  \\ 0         \\ 1 -\eta_1 \end{bmatrix},
		       \begin{bmatrix}  0  \\ \eta_1\eta_2 \end{bmatrix} \right),\\
  \D_7(\eta) &= \left( \begin{bmatrix} -\eta_1\eta_2  \\ 0     \\ \eta_1(1- \eta_2) \end{bmatrix},
		       \begin{bmatrix} 0               \\ 1  \end{bmatrix} \right), 
  &&\D_8(\eta)  = \left( \begin{bmatrix} -\eta_1\eta_2  \\ 0     \\ 1 - \eta_1\eta_2 \end{bmatrix},
		       \begin{bmatrix} 0               \\ \eta_1  \end{bmatrix} \right),
  &&\D_9(\eta) = \left( \begin{bmatrix} 0               \\ 0      \\ \eta_2  \end{bmatrix},
		       \begin{bmatrix} -\eta_1  \\ 1 - \eta_1  \end{bmatrix} \right).   
\end{alignat*}
 The integral can be simplified similarly to the case described in Section~\ref{sec:duffy3d3dface}.
 This results in 
 \begin{equation}\label{eq:Ergebnis_Duffy_3d2d_Face}
I_{t,\tau} = \prod_{j=0}^2 \frac{1}{5-j-2s} \,  \int_{[0,1]^2} \tilde k_{2,F}(\eta)\ud \eta,\quad
\tilde k_{2,F}(\eta):=\sum_{m=1}^9\, \tilde k_2(\D_m(\eta))\, \deta\J_m(\eta). 
\end{equation}
Hence, only a two-dimensional integral has to be integrated numerically instead of a five-dimensional one.
The regularity of the remaining integrand will be proved in Lemma~\ref{lem:analyt_3d2d_Face}.
\section{Error Estimates for the Integrals} \label{sec:Error_Estimates}
\subsection{Derivative free error estimates}
All previous transformations result in integrals over multi-dimensional unit cubes.
In the following, we present quadrature rules and the corresponding error estimates.
 We follow the approach of Sauter and Schwab~\cite{Sauter} and use derivate-free estimates, which 
are based on the work of Davis~\cite{Davis}.
By $\mathcal{E}^\rho_{a,b} \subset \mathbb{C}$ we denote the closed ellipse with focal points~$a,b\in\R$. 
$\rho =: \bar{a} + \bar{b}$ refers to the sum of the semimajor half-axis $\bar{a} > (b-a)/2$ and 
and the semininor half-axis~$\bar{b}$. The standard ellipse~$\mathcal{E}_{0,1}^\rho$ is abbreviated with~$\mathcal{E}_\rho$.
Furthermore, let $k\in \N$ and let $f:[0,1]^k\to \C$ be a function. 
Then we denote by $I[f]$ the integral of $f$ over $[0,1]^k$ and
by $Q^n[f]$ its approximation with a tensor Gauss quadrature of order~$n \in \N^k$, i.e.,
in dimension $i\in\{1,\dots,k\}$ we use a Gaussian quadrature rule with $n_i$ points.
Moreover, we define the quadrature error $E^n[f]:= |I[f]- Q^n[f]|$.

\begin{theorem}\label{thm:error_int_1d}
Let $f:[0,1]\to \mathbb{C}$ be analytic with an analytic extension $f^\star$ to $\mathcal{E}_\rho$ for some $\rho > \frac{1}{2}$.
Then
\[        E^n[f] \leq c\,(2\rho)^{-2n}\, \max_{z\in \partial \mathcal{E}_\rho} |f^\star(z)|.
\]
\end{theorem}

In order to generalize the previous theorem to a tensor Gauss quadrature, we need the following definition.
\begin{definition}\label{def:int} 
  Let $X := [0,1]^k \subset \R^k$.
  A function $f: X \to \C$ is called \emph{componentwise analytic} if for $i=1,\dots,k$ there exists $\rho_i > 1/2$
  such that for all $x \in X$ the function
  \begin{equation*}
         f_{i, x}: [0,1] \to \C, \quad f_{i, x}(t) := f(x_1,\ldots,x_{i-1},t,x_{i+1},\ldots,x_k),
  \end{equation*}
  can be extended to an analytic function $f^*_{i, x}: \mathcal{E}_{\rho_i} \to \C$.
\end{definition}
With this definition we can extend Theorem~\ref{thm:error_int_1d} to the multi-dimensional case.

\begin{theorem} \label{thm:error_int_dd}
  Let $f:[0,1]^k \to \C$ be componentwise analytic and let $\rho\in \R^k$ be as in Definition~\ref{def:int}. 
  Then the quadrature error with $n_i$ nodes in dimension~$i=1,\dots,k$ satisfies
  \begin{equation*}
    E^n[f]  \leq \sum_{i=1}^k \max_{ x \in [0,1]^k}  E^{n_i}[f_{i,x}] 
           \leq \sum_{i=1}^k c_i\,(2\rho_i)^{-2n_i} \max_{x \in [0,1]^k} \max_{z \in \partial \mathcal{E}_{\rho_i}} \lvert f^*_{i,x}(z) \rvert.
  \end{equation*}
  The constants $c_i$, $1\leq i \leq k$, are as in Theorem~\ref{thm:error_int_1d}.
\end{theorem}
The proofs of Theorem~\ref{thm:error_int_1d} and~\ref{thm:error_int_dd} can be find in \cite{Sauter}.
Our main goal is to verify the requirements of Theorem~\ref{thm:error_int_dd} 
for each of the cases in Section~\ref{sec:Duffy_Trafo}.

\subsection{Preparation}

\subsubsection{Analyticity of the integrands} \label{sec:Analyticity}
In order to apply Theorem \ref{thm:error_int_dd} to our cases \eqref{int_3d3d} and \eqref{int_3d2d},
we have to prove that for each case the integrand is componentwise analytic.
For this we also need assumptions on the geometry.

\begin{assumption}
 \begin{enumerate} Let us assume that
  \item For all $t \in \T$ there is $c_A>0$ and $x_0\in\bar t$ such that
          $|x-y| \geq c_{A}\, ( |x-x_0|+ |x_0-y|)$ for all $x\in\Omega\setminus t$ and $y\in t$.
  \item For all $t_1, t_2 \in \T$ whose intersection consists of a common vertex $p$,
        there exists $c_{A}>0$ such that $|x-y| \geq c_{A}\,( |x-p| + |p-y|)$ for $x \in t_1$, $y\in t_2$.
  \item For all $t_1, t_2 \in \T$ with exactly one common edge $E:=\bar t_1 \cap \bar t_2$, 
        there exists $c_{A}>0$ and a point $p \in E$ such that
        $|x-y| \geq c_{A}\,( |x-p| + |p-y|)$ for all $x \in t_1$, $y\in t_2$.  
  \item For all $t_1, t_2 \in \T$ with exactly one common face $F:=\bar t_1 \cap \bar t_2$, 
        there exists $c_{A}>0$ and a point $p \in F$ such that
        $|x-y| \geq c_{A}\,( |x-p| + |p-y|)$ for all $x \in t_1$, $y\in t_2$.        
 \end{enumerate} \label{assumpt:xi}
\end{assumption}
A similar assumption is made for the combination of tetrahedra in $\T$ and panels in $\Pa$. 
\begin{assumption} 
 \begin{enumerate} Let us assume that
  \item For all $t \in \T$ there is $c_{A}>0$ and $x_0 \in \bar t$ such that 
		$|x-y| \geq c_{A}\, ( |x-x_0|+ |x_0-y|)$ for all $y \in \tau \subset \Pa\backslash \bar{t}$ and $x \in t$.
  \item For all $t \in \T$ and $\tau \in\Pa$ whose intersection consists of a common vertex $p$,
        there is $c_{A}>0$  such that
        $|x-y| \geq c_{A}( |x-p| + |p-y|)$ for all $x \in t$, $y\in \tau$.
  \item For all $t \in \T$ and $\tau \in\Pa$ with exactly one common edge $E:=\bar t \cap \bar \tau$, 
        there exists $c_{A}>0$  and a point $p \in E$ such that
        $|x-y| \geq c_{A}( |x-p| + |p-y|)$ for all $x \in t$, $y\in \tau$.
  \item For all $t \in \T$ and $\tau \in\Pa$ with exactly one common face $F:=\bar t \cap \bar \tau = \bar \tau,$ 
        there is $c_{A}>0$  and a point $p \in \bar \tau$ such that
        $|x-y| \geq c_{A}( |x-p| + |p-y|)$ for all $x \in t$, $y\in \tau$.                  
 \end{enumerate} 
\end{assumption}
In addition, we also need information about the scaling behavior of the linear mappings $\chi_t$ and $\chi_\tau$.
\begin{lemma}[Proposition 5.3.8 in \cite{Sauter}]\label{lem:mapPanel}
 Let $M_\tau$ be as in \eqref{eq:chi_tau} and let $\lambda_{\max}$ and $\lambda_{\min}$
be the largest and the smallest eigenvalue of $M_\tau^TM_\tau$, respectively.
 Then
 \[ c\, h_\tau^2 \leq \lambda_{\min} \leq \lambda_{\max} \leq 2\, h_\tau^2 \quad \text { and } \quad
    c\, h_\tau^2 \leq \sqrt{\deta M_\tau^T M_\tau} \leq C\, h_\tau^2 \]
 with constants $c,C > 0 $ that depend only on the minimal interior angle $\theta_\tau$ of the panel $\tau$.   
\end{lemma}

\begin{lemma}\label{lem:mapTetra}
 Let $M_t$ be as in \eqref{eq:chi_t} and let $\lambda_{\max}$ and $\lambda_{\min}$ be the largest and smallest eigenvalue of $M_t^TM_t$, respectively.
 Then 
 \[ c\, h_t^2 \leq \lambda_{\min} \leq \lambda_{\max} \leq 3\, h_t^2 \quad \text { and } \quad
    c\, h_t^3 \leq \sqrt{\deta M_t^TM_t} \leq C\, h_t^3 \]
 with constants $c,C > 0 $ that depend only on the minimal interior angle $\theta_t$ of the tetrahedron $t$.   
\end{lemma}

\begin{proof}
 We set $v_1:= b_t - a_t$, $v_2:= c_t -b_t$, and $v_3 := d_t-b_t$ and denote by $\rho_t$ the radius of the insphere of $t$. 
 The upper bound on the eigenvalues of $M_t^TM_t$ results from the linearity of the scalar product and Young's inequality
 \[
    (M_t \xi, M_t \xi) = \sum_{i,j=1}^3 \xi_i\,\xi_j\, (v_i,v_j)
                       \leq 3\, \sum_{i=1}^3 \xi_i^2\, |v_i|^2 \leq 3\,  h_t^2\, |\xi|^2,   \quad \xi \in \R^3.
                       \]         
 According to \cite{Dziuk}, for the lower bound it holds that $\lambda_{\min} \geq \rho_t/h_{\tilde t}$. 
 Using elementary properties of the tetrahedron, the estimates in \cite{Richardson} can be adjusted to our needs:
 \[
  \lambda_{\min} \geq \rho_t\,h_{\tilde t}^{-1} = 3h_{\tilde t}^{-1} \frac{|t|}{|\partial t|} 
                 \geq 3h_{\tilde t}^{-1} \sin^2(\theta_t) \frac{|v_1||v_2||v_3|}{|\partial t|}
                 \geq \tilde c\, \sin^2(\theta_t)\, h_t.
 \]
 The last step is due to the assumption that the tetrahedra do not degenerate.
\end{proof}
Since we assumed in \eqref{eq:theta_T} that the minimum interior angles of the tetrahedra are bounded below by~$\theta_\T$, 
the Lemmas \ref{lem:mapPanel} and \ref{lem:mapTetra} show that there exists a constant $c_\T >0$
such that for each $t \in \T$ and for each $\tau \in \Pa$ it holds that
\begin{equation}\label{eq:c_T}
   |M_t x| \geq c_\T h_t\, |x|, \quad x \in \R^3,  \quad \text{ and } \quad 
   |M_\tau y| \geq c_\T h_\tau\, |y|, \quad y \in \R^2.
\end{equation}

\paragraph{Tetrahedron and tetrahedron}
With these assumptions we can prove that the integrands after the Duffy transformation are analytic in the corresponding integration domain.
The direct consequence of this is also that the Duffy transformation has lifted the singularity.
\begin{lemma}\label{lem:analyt_3d3d_Tetra} 
   The integrands $\tilde k_{1,T}$ in \eqref{eq:Ergebnis_Duffy_3d3d_Tetra}, $\tilde k_{1,F}$ in \eqref{eq:Ergebnis_Duffy_3d3d_Face},
   $\tilde k_{1,E}$, in \eqref{eq:Ergebnis_Duffy_3d3d_Edge}, and $\tilde k_{1,V}$ in \eqref{eq:Ergebnis_Duffy_3d3d_Vertex}
   are componentwise analytic for $s\in(0,1)$.
\end{lemma}
\begin{proof}
  Since in the case of the integrand $\tilde k_{1,T}$ the two tetrahedra $t_1$ and $t_2$ are identical, we drop the index. 
  For each sub-domain, we have to consider three terms: the Jacobi determinant of the Duffy transformation, the product of the linear basis functions and the denominator.
  The first two can obviously  be   extended analytically with respect to all variables in a complex neighborhood of $[0,1]^2$ as 
  they are polynomials w.r.t.\ $\eta$.  For the denominator we obtain
  \begin{equation}\label{eq:Dtt}
    | M_t (\D_m^2(\eta) - \D_m^1(\eta))| \geq c_\T h_t |\D_m^2(\eta) - \D_m^1(\eta)| =: c_\T h_t \,\dist_t(\D_m(\eta)),\quad m=1,\ldots,9,
  \end{equation}
where $c_\T$ is defined in \eqref{eq:c_T}.
  Therefore, only the distance has to be examined for each domain:
  \begin{alignat*}{2}
     \dist_t^2(\D_1(\eta)) &= \eta_1^2(1 + \eta_2^2) + (1-\eta_1)^2 > 0,   &&\quad       
     \dist_t^2(\D_2(\eta))^2 = \eta_1^2\eta_2^2 + 1 + \eta_1^2(1-\eta_2)^2 > 0, \\
     \dist_t^2(\D_3(\eta)) &= \eta_1^2( 1 + \eta_2^2) +  (1-\eta_1\eta_2)^2 > 0, &&\quad
     \dist_t^2(\D_4(\eta)) = \eta_1^2( 1 + \eta_2^2) +  (1-\eta_1\eta_2)^2 > 0, \\
     \dist_t^2(\D_5(\eta)) &= 1 + \eta_1^2\eta_2^2 + \eta_1^2(1-\eta_2)^2 > 0, &&\quad
     \dist_t^2(\D_6(\eta)) = \eta_1^2( 1 + \,\eta_2^2) +  (1-\eta_1)^2 > 0, \\
     \dist_t^2(\D_7(\eta)) &= \eta_1^2\eta_2^2 + \eta_1^2( 1 - \eta_2)^2 +  1 > 0, &&\quad     
     \dist_t^2(\D_8(\eta)) = \eta_1^2\eta_2^2 + \eta_1^2 ( 1 - \eta_2)^2 +  (1-\eta_1)^2 > 0, \\
     \dist_t^2(\D_9(\eta)) &= \eta_1^2 + (1-\eta_2)^2 + \eta_2^2    > 0.
  \end{alignat*}
  Notice, due to symmetry it is enough to only consider the first $9$ sub-domains.
  Since the denominator has no zeros in $[0,1]^2$ and consists of analytic functions,
  the whole integrand $\tilde k_{1,T}$ can be analytically extended with respect to all variables in a complex neighborhood of $[0,1]^2$ for $s\in(0,1)$.

  The main difference for the integrand $\tilde k_{1,F}$ is the estimation of the denominator.
  The intersection between the tetrahedra $t_1$ and $t_2$ is a common face $\tau=\bar t_1 \cap \bar t_2$. 
  Without loss of generality we assume that the parametrizations $\chi_{t_1}$ and $\chi_{t_2}$ satisfy the relation
  \[
      \chi_{t_1}( \kappa) = \chi_{t_2}(\kappa)
  \]
  with $\kappa:=( \kappa_1, \kappa_2, 0)^T$ for all $\kappa_1,\kappa_2 \in [0,1]$.
  By Assumption \ref{assumpt:xi}, there exists a point $p$ on the common face $\tau$ with $p = \chi_{t_1}(\kappa)$, and it holds that
  \begin{align*}
     |\chi_{t_2}(\D_m^2(\eta)) - \chi_{t_1}(\D_m^1(\eta))|                  
                     &\geq c_{A}\, (| M_{t_2}(\D_m^2(\eta) - \kappa) | + | M_{t_1}(\kappa - \D_m^1(\eta))|) \\
                     &\geq c\,c_{A}\,c_\T h_t \left[ \sum_{i=1}^2 (\D_{m,i}^2(\eta)-\kappa_i)^2 + (\kappa_i - \D_{m,i}^1)^2 + \D_{m,3}^1(\eta)^2 + \D_{m,3}^2(\eta)^2\right]^{1/2} \\
                     &\geq \frac{c}{2} h_t \left[ \sum_{i=1}^2 (\D_{m,i}^2(\eta)-\D_{m,i}^1(\eta))^2 + \D_{m,3}^1(\eta)^2 + \D_{m,3}^2(\eta)^2\right]^{1/2} \\
                     &=: \frac{c}{2}h_t\, \dist_\tau(\D_m(\eta)),
  \end{align*}
  where $c_\T$ is defined in \eqref{eq:c_T}.
  Therefore, only the $\dist_\tau(\D_m(\eta))$ has to be examined for each domain.
  Since the treatment of $\dist_\tau(\D_m(\eta))$ is analogous to the treatment of $\dist_t(\D_m(\eta))$, we consider only the first sub-domain:
  \[
     |\dist_\tau(\D_1(\eta))|^2 = \eta_1^2( 1+ \eta_2^2[ \eta_3^2 + (1-\eta_3)^2]) + (1-\eta_1)^2  > 0.  
  \]
  The remaining sub-domains can be estimated analogously.
  Since the denominator has no zeros in~$[0,1]^3$ and consists of analytic functions,
  the whole integrand $\tilde k_{1,F}$ can be analytically extended with respect to all variables in a complex neighborhood of $[0,1]^3$ for $s\in(0,1)$.

  Similar arguments apply in the case of the $\tilde k_{1,E}$ and $\tilde k_{1,V}$.
\end{proof}

\paragraph{Tetrahedron and Panel}
In this section, the intersection between a tetrahedron and a panel is investigated.
The proofs of the lemmas are omitted as the structure and the ideas of the proofs are analogous to the proofs for $\tilde k_{1,F}$ and $\tilde k_{1,V}$.

\begin{lemma}\label{lem:analyt_3d2d_Face} 
The integrands $\tilde k_{2,F}$ in \eqref{eq:Ergebnis_Duffy_3d2d_Face}, $\tilde k_{2,E}$ in \eqref{eq:Ergebnis_Duffy_3d2d_Edge}, and $\tilde k_{2,V}$ in \eqref{eq:Ergebnis_Duffy_3d2d_Vertex} are componentwise analytic for $s\in(0,1)$.
\end{lemma}

\subsubsection{Estimation of the integrals}
Since it was shown in Section~\ref{sec:Analyticity} that all regularized integrands are analytic, 
the requirements of Theorem~\ref{thm:error_int_dd} are fulfilled. 
In order to apply the latter theorem, the sum of the half-axes $\rho$ of the ellipse has to be estimated for each case.
Additionally, we need a notation to describe the analytic extension onto the ellipses w.r.t.\ to the integration domain.
We define for $\rho>0$ and $i \in \{1,\ldots,k\}$
\[
    \mathcal{E}_{\rho}^{(i)} := [0,1]^{i-1} \times \mathcal{E}_{\rho} \times [0,1]^{k-i} \subset \R^k.
\]

\paragraph{Tetrahedron and Tetrahedron}
First, we consider the different interactions between two tetrahedra. 
In contrast to \cite{Sauter}, instead of using spherical coordinates we use the Duffy transformation directly.
The next lemma gives the desired error estimate for the case that the two tetrahedra are identical.
\begin{lemma}\label{lem:error_est_integrand_3d3d_Tetra}
 There exists a constant $\rho> \frac12$ that depends only on $\theta_\T$ such that the integrand $\tk_{1,T}$ in~\eqref{eq:Ergebnis_Duffy_3d3d_Tetra}
 can be analytically extended to $\bigcup_{i=1}^2 \dellipse_{\rho}^{(i)}$.
 It holds that
 \[
    \sup_{\eta \in \dellipse_{\rho}^{(i)}}   |\tk_{1,T}(\eta)| \leq C\, h^{3-2s}, \quad i=1,2.
 \]
\end{lemma}

\begin{proof}
 Since the two tetrahedra are identical, we drop the index.
 From Lemma \ref{lem:analyt_3d3d_Tetra} we know that $\tk_{1,T}$ is analytic in a complex neighborhood of $[0,1]^2$.
 Therefore, $\rho>1/2$ exists such that $\tk_{1,T}$ can be analytically extended to $\bigcup_{i=1}^2 \dellipse_{\rho}^{(i)}$.
 We show that the choice of $\rho$ is independent of $h$.
 In the proof of Lemma~\ref{lem:analyt_3d3d_Tetra}, we have seen that denominator of
 \[
  \tk_{1,T}(\eta) = \eta_1 \sum_{m=1}^{9} \frac{[\alpha_{i,1:3}^T\, M_t(\D_m^2(\eta)-\D_m^1(\eta))] \, [\alpha_{j,1:3}^T\, M_t(\D_m^2(\eta)-\D_m^1(\eta))]}{|h_t^{-1}M_t(\D_m^2(\eta)-\D_m^1(\eta))|^{3+2s}}\,h_t^{-3-2s}|\deta M_t|^2.
 \]
 is responsible for the size of the complex neighborhood.
 Due to Lemma~\ref{lem:mapTetra} and~\eqref{eq:c_T}, it holds that
 \[
    | h_t^{-1} M_t x | \geq c_\T |x|,\quad x\in\R^3,
 \]
 where the positive constant $c_\T$ does not depend on $h_t$.
 By combining this with \eqref{eq:Dtt}, we obtain
 \[
  |h_t^{-1}M_t(\D_m^2(\eta)-\D_m^1(\eta))| \geq c_\T |\D_m^2(\eta)-\D_m^1(\eta)| > 0, \quad  \eta \in [0,1]^2, \quad m=1,\ldots,9.
 \]
 Since the $\D_m$ do not depend on $h_t$, it follows that the size of the complex neighborhood is independent of $h_t$ and thus also independent of $h$; see~\eqref{assum:quasi_uniform_h}. 
 This implies that the choice of $\rho$ is independent of $h$.
 We continue with the estimation of the integrand.
 By considering the scaling of the linear mapping~$\chi_t$, we obtain
 \[ | h_t^{-1}\, M_t x |
       \leq |x_1|\, \frac{| b_t - a_t |}{h_t} 
         +  |x_2|\, \frac{| c_t - b_t |}{h_t}
         +  |x_3|\, \frac{| d_t - b_t |}{h_t} 
       \leq \sqrt{3} | x|.
 \]
 Due to the scaling of $\alpha_{\star,1:3}$ with $\star \in \{i,j\}$, it holds
 \[
    | \alpha_{\star,1:3}^T\, M_t x| \leq C\, |x|,
 \]
 where the positive constant $C$ is independent of $h$.
 By combining these results, we obtain for $i=1,2$ and for $m=1,\ldots,9$, that
 \[
    \sup_{\eta \in \dellipse_{\rho}^{(i)}} \left| \eta_1\, \frac{[\alpha_{i,1:3}^T\, M_t(\D_m^2(\eta)-\D_m^1(\eta))] \, [\alpha_{j,1:3}^T\, M_t(\D_m^2(\eta)-\D_m^1(\eta))]}{|h_t^{-1}\,M_t(\D_m^2(\eta)-\D_m^1(\eta))|^{3+2s}} \right|\leq \tilde c
 \]
 with $\tilde c>0$ independent of~$h_t$. 
 Using Lemma~\ref{lem:mapTetra}, this leads to 
 \[
   \sup_{\eta \in \dellipse_{\rho}^{(i)}} | \tk_{1,T}(\eta)| \leq \tilde c \sum_{m=1}^{9} h_t^{-3-2s}\, |\deta M_t|^2 \leq 9\, C\, h^{3-2s},\quad i=1,2.
 \]
  \end{proof}

  Next, the case of a singularity along one face is considered.

 \begin{lemma}
  There exists a constant $\rho > 1/2 $ that depends only on $\theta_{\T}$  
  such that the integrand $\tk_{1,F}$ in~\eqref{eq:Ergebnis_Duffy_3d3d_Face}
  can be analytically extended to $\bigcup_{i=1}^3 \dellipse_{\rho}^{(i)}$.
  It holds that
 \[
       \sup_{\eta \in \dellipse_{\rho}^{(i)}} |\tk_{1,F}(\eta)| \leq C\, h^{3-2s}, \quad i=1,2,3.
 \]
\end{lemma}
  
  Now we examine what happens when the singularity is along a corner.
  
 \begin{lemma}
 There exists a constant $\rho>1/2 $ that depends only on $\theta_{\T}$  
  such that the integrand $\tk_{1,E}$ in~\eqref{eq:Ergebnis_Duffy_3d3d_Edge}
  can be analytically extended to $\bigcup_{i=1}^4 \dellipse_{\rho}^{(i)}$.
  It holds that
 \[
    \sup_{\eta \in \dellipse_{\rho}^{(i)}} |\tk_{1,E}(\eta)| \leq C\, h^{3-2s},\quad i=1,\dots,4.
 \]
\end{lemma}

 Here, the singularity at a common vertex is examined.
  
  \begin{lemma}\label{lem:error_est_integrand_3d3d_Vertex}
 There exists a constant $\rho > 1/2 $ that depends only on $\theta_{\T}$   
  such that the integrand $\tk_{1,V}$ in~\eqref{eq:Ergebnis_Duffy_3d3d_Vertex}
  can be analytically extended to $\bigcup_{i=1}^5 \dellipse_{\rho}^{(i)}$.
  It holds that
 \[
    \sup_{\eta \in \dellipse_{\rho}^{(i)}} |\tk_{1,V}(\eta)| \leq C\, h^{3-2s},\quad i=1,\ldots,5.
 \]
\end{lemma}
The proofs of the Lemmas~\ref{lem:error_est_integrand_3d3d_Tetra} to~\ref{lem:error_est_integrand_3d3d_Vertex} are structurally the same.
Therefore, they are omitted.

Finally, the case is investigated in which the tetrahedra have a positive distance
  \[ d_{t_1,t_2} := \dist(t_1,t_2) := \inf_{(x,y)\in t_1\times t_2} |x-y | > 0. \]
 The transformation 
  \[
      \quad \D(\eta):=( (\D_1(\eta_{1:3}),\D_2(\eta_{4:6})) = \left( \begin{bmatrix}  \eta_1\\ \eta_1\eta_2\\ \eta_1\eta_2\eta_3 \end{bmatrix},
                                                                     \begin{bmatrix}  \eta_4\\ \eta_4\eta_5\\ \eta_4\eta_5\eta_6 \end{bmatrix}\right),\quad\eta\in\R^6,
  \]
  is analytic and we have $\deta\J(\eta)  = \eta_1^2\eta_2\eta_4^2\eta_5$.
 Applying this transformation to~\eqref{eq:TT_start}, we obtain
  \begin{equation}\label{eq:Ergebnis_Duffy_3d3d_Dist}
   I_{t_1,t_2} = \int_{[0,1]^6} \tk_{1,D}(\eta) \ud \eta
  \end{equation}
  with
  \[
     \tk_{1,D}(\eta) 
        := \frac{\beta_{ij}(\D_1(\eta),\D_2(\eta))}{|\chi_{t_2}(\D_2(\eta))-\chi_{t_1}(\D_1(\eta))|^{3+2s}}\, |\deta M_{t_1}|\,|\deta M_{t_2}|\, \deta\J(\eta)
  \]
  and 
  \[
     \beta_{ij}(\D_1(\eta),\D_2(\eta)) := 
         \begin{cases}
                      - \fie_i(\chi_{t_1}(\D_1(\eta))\, \fie_j(\chi_{t_2}(\D_2(\eta)), & t_1 \in \supp\, \fie_i,\, t_2 \in \supp\, \fie_j, \\
                      - \fie_j(\chi_{t_1}(\D_1(\eta))\, \fie_i(\chi_{t_2}(\D_2(\eta)), & t_2 \in \supp\, \fie_i,\, t_1 \in \supp\, \fie_j.
                 \end{cases}
  \]
 We focus on the almost singular cases, i.e.\ $ d_{t_1,t_2}\approx h$, as only those cause difficulties in the numerical integration.
  \begin{lemma}\label{lem:error_est_integrand_3d3d_dist}
 There exists a constant $\rho > \frac12$ that depends only on $\theta_\T$ such that $\tk_{1,D}$  
 can be analytically extended to $\bigcup_{i=1}^6 \dellipse_{\tilde\rho}^{(i)}$ .
 It holds that
 \[
    \sup_{\eta \in \dellipse_{\tilde \rho}^{(i)}} |\tk_{1,D}(\eta)| \leq C\, h^{3-2s},\quad i=1,\ldots,6.                                                                    
 \]
\end{lemma}
We omit the proof of Lemma~\ref{lem:error_est_integrand_3d3d_dist} as it is analogous to the proof of Lemma~\ref{lem:error_est_integrand_3d3d_Tetra}.

\paragraph{Tetrahedron and Panel}
 Next, we consider the interaction between a tetrahedron and a panel.
 First, the case of a singularity along a common face is considered, i.e.\ the panel is a face of the tetrahedron.

  \begin{lemma}\label{lem:error_est_integrand_3d2d_Face}
 There exists a constant $\rho > \frac12$ that depends only on $\theta_\T$ 
 such that the integrand $\tk_{2,F}$ in~\eqref{eq:Ergebnis_Duffy_3d2d_Face}
 can be analytically extended to $\bigcup_{i=1}^2 \dellipse_{\rho}^{(i)}$.
 It holds that
 \begin{align*}
    \sup_{\eta \in \dellipse_{\rho}^{(i)}} |\tk_{2,F}(\eta)| &\leq C\, h^{3-2s}, \quad i=1,2.
 \end{align*}
\end{lemma}
 
 Second, the case of the singularity along an edge is investigated.
  
 \begin{lemma}
 There exists a constant $\rho > \frac12$ that depends only on $\theta_\T$ 
 such that the integrand $\tk_{2,E}$ in~\eqref{eq:Ergebnis_Duffy_3d2d_Edge}
 can be analytically extended to $\bigcup_{i=1}^3 \dellipse_{\rho}^{(i)}$.
 It holds that
 \begin{align*}
    \sup_{\eta \in \dellipse_{\rho}^{(i)}} |\tk_{2,E}(\eta)| &\leq C\, h^{3-2s}, \quad i=1,2,3.
 \end{align*}
\end{lemma}

  Next, the singularity at a common vertex is studied.
  
  \begin{lemma}\label{lem:error_est_integrand_3d2d_Vertex}
 There exists a constant $\rho > \frac12$ that depends only on $\theta_\T$ 
 such that the integrand $\tk_{2,V}$ in~\eqref{eq:Ergebnis_Duffy_3d2d_Vertex}
 can be analytically extended to $\bigcup_{i=1}^4 \dellipse_{\rho}^{(i)}$.
 It holds that
 \begin{align*}
    \sup_{\eta \in \dellipse_{\rho}^{(i)}} |\tk_{2,V}(\eta)| &\leq C\, h^{3-2s}, \quad i=1,\ldots,4.
 \end{align*}
\end{lemma}
The proofs of the Lemmas~\ref{lem:error_est_integrand_3d2d_Face} to~\ref{lem:error_est_integrand_3d2d_Vertex} are omitted as they are analogous to the proof of Lemma~\ref{lem:error_est_integrand_3d3d_Tetra}.

 Finally, the case is considered in which the tetrahedron and the panel have a positive distance
  \[ d_{t,\tau} := \dist(t,\tau) := \inf_{(x,y)\in t\times \tau} |x-y | > 0, \]
 The transformation 
  \[
      \quad \D(\eta):=( (\D_1(\eta_{1:3}),\D_2(\eta_{4:5})) = \left( \begin{bmatrix}  \eta_1\\ \eta_1\eta_2\\ \eta_1\eta_2\eta_3 \end{bmatrix},
                                                                     \begin{bmatrix}  \eta_4\\ \eta_4\eta_5 \end{bmatrix}\right),\quad \eta\in\R^5,
  \]
  is analytic and we have $\deta\J(\eta)  = \eta_1^2\,\eta_2\,\eta_4$.
  Applying this transformation to~\eqref{eq:TP_start}, we obtain
  \begin{equation}\label{eq:Ergebnis_Duffy_3d2d_Dist}
     I_{t,\tau} :=  \int_{[0,1]^5} \tk_{2,D}(\D(\eta)) \ud \eta
   \end{equation} 
  with
  \[
     k_{2,D}(\D(\eta)) 
        := \fie_i(\chi_t(\D_1(\eta))\, \fie_j(\chi_t(\D_1(\eta))\, \frac{[\chi_\tau(\D_2(\eta))-\chi_t(\D_1(\eta))]^T n}{|\chi_{t_2}(\D_2(\eta))-(\chi_{t_1}(\D_1(\eta))|^{3+2s}}\, |M_t|\,|M_\tau|\, \deta\J(\eta).
  \]
  Notice that $\tau \in \partial( \supp\fie \cup \supp\fie)$ and therefore, it holds that $d_{t,\tau} \approx h$.

  \begin{lemma}\label{lem:error_est_integrand_3d2d_dist}
   There exists a constant $\rho > 1/2$ that depends only on $\theta_\T$ 
   such that the integrand $\tk_{2,D}$ can be analytically extended to $\bigcup_{i=1}^5 \dellipse_\rho^{(i)}$.
   It holds that
   \[
    \sup_{\eta\in\dellipse_\rho^{(i)}} |\tk_{2,D}(\eta)| \leq C\, h^{3-2s}, \quad i=1,\ldots,5.
   \]
  \end{lemma}
 Since the proof of Lemma~\ref{lem:error_est_integrand_3d2d_dist} is analogous to the proof of Lemma~\ref{lem:error_est_integrand_3d3d_dist},
 it is omitted.

\subsubsection{Error estimates for the integrals}
First, we start with the integrals of the type \eqref{int_3d3d}.

\begin{theorem} \label{thm:estimates_int_sing_3d3d}
 The approximations of the integrals~\eqref{eq:Ergebnis_Duffy_3d3d_Dist}, \eqref{eq:Ergebnis_Duffy_3d3d_Vertex}, \eqref{eq:Ergebnis_Duffy_3d3d_Edge}, \eqref{eq:Ergebnis_Duffy_3d3d_Face}, and~\eqref{eq:Ergebnis_Duffy_3d3d_Tetra} by means of tensor Gauss quadrature converge exponentially with respect to~$n$
 \begin{align}
   \lvert E^n[\tilde k_{1,\star}] \rvert &\leq C\, h^{3-2s}\, (2\rho_\star)^{-2n}
 \end{align}
 with $\rho_\star>1/2$ and $\star \in \{D,V,E,F,T\}$. 
\end{theorem}

\begin{proof}
 Lemma~\ref{lem:analyt_3d3d_Tetra} shows that the requirement of Theorem~\ref{thm:error_int_dd} is fulfilled for all singularity types 
 and Lemma~\ref{lem:error_est_integrand_3d3d_Tetra} to Lemma~\ref{lem:error_est_integrand_3d3d_Vertex} give the corresponding estimates for the integrands.
 Therefore, we can apply Theorem~\ref{thm:error_int_dd} and obtain the desired results.
 
 Since there is a positive distance between the tetrahedra, the integrand $\tk_{1,D}$ is obviously analytic.
 Therefore, Theorem~\ref{thm:error_int_dd} can be applied and by combining it with Lemma~\ref{lem:error_est_integrand_3d3d_dist} 
 we obtain the desired estimate.
\end{proof}
Notice that for the integral~\eqref{eq:Ergebnis_Duffy_3d3d_Dist} we have a worst-case estimation, since we assume that $d_{t_1,t_2} \approx h$.  

Second, we consider the integrals of type~\eqref{int_3d2d}.

\begin{theorem} \label{thm:estimates_int_sing_3d2d}
 The approximation of the integrals~\eqref{eq:Ergebnis_Duffy_3d2d_Dist}, \eqref{eq:Ergebnis_Duffy_3d2d_Vertex}, \eqref{eq:Ergebnis_Duffy_3d2d_Edge}, and~\eqref{eq:Ergebnis_Duffy_3d2d_Face} 
 by means of tensor Gauss quadrature converge exponentially with respect to $n$
 \begin{align}
   \lvert E^n[\tilde k_{2,\star}] \rvert &\leq C\, h^{3-2s}\, (2\rho_\star)^{-2n}
 \end{align}
 with $\rho_\star>1/2$ and $\star \in \{D,V,E,F\}$. 
\end{theorem}

\begin{proof}
 First, we consider the singularity cases.
 Lemma~\ref{lem:analyt_3d2d_Face} shows that the requirements of Theorem~\ref{thm:error_int_dd} are fulfilled for all singularity types 
 and Lemma~\ref{lem:error_est_integrand_3d2d_Face} to Lemma~\ref{lem:error_est_integrand_3d2d_Vertex} provides the corresponding estimates.
 Therefore, we can apply Theorem~\ref{thm:error_int_dd} and obtain the desired results.
 
 For the non-singular case it is obvious that the integrand $\tk_{2,D}$ is analytic. 
 Therefore, the requirements of Theorem~\ref{thm:error_int_dd} are fulfilled 
 and Lemma~\ref{lem:error_est_integrand_3d2d_dist} provides the corresponding estimates.
\end{proof}

\subsection{Rules for the number of Gauss points}

Since we now have error estimates for the numerical integration, it can be investigated how this error proceeds through the bilinear form. This was first done by \cite{AiCl2018}. 

\begin{theorem} \label{error_biliearform}
  Let $n_1$ and $n_2$ be the quadrature orders used for touching pairs of tetrahedra and tetrahedron and  panel, respectively.
  Denote by $a_Q$ the resulting approximation to the bilinear form $a$.  
  Then the consistency error due to the quadrature is bounded by
  \begin{equation}
     \lvert a(u,v) - a_Q(u,v)\rvert \leq C(E_1+E_2) \,\norm{u}_{L^2(\Omega)} \,\norm{v}_{L^2(\Omega)}, \quad u,v \in V_h,
  \end{equation}
  where 
  \[
    E_{1}  = h^{-3-2s}\, (2\rho_1)^{-2n_1} \quad \text{ and } \quad E_{2}  =  h^{-2-2s}\, (2\rho_2)^{-2n_2}
  \]
   with constants $\rho_j > 1/2$, $j=1,2$.
\end{theorem}
\begin{proof}
 In \cite{AiCl2018} it was shown that
 \begin{align*}
   \lvert a(u,v) - a_Q(u,v) \rvert 
       &\leq C h^{-2d} \left[ \max_{t_1,t_2 \in \T}\,\max_{i,j\in I_{t_1,t_2}} | E^{t_1,t_2}_{i,j} | + h\, \max_{t\in\T, \tau\in\Pa_{\partial\Omega}}\,\max_{i,j\in I_t} | E^{t,\tau}_{i,j} | \right] 
            \lVert u \rVert_{L^2(\Omega)} \, \lVert v \rVert_{L^2(\Omega)},     
\end{align*}  
where $E^{t_1,t_2}_{i,j}$ and $E^{t,\tau }_{i,j}$ denote the error between the integrals $a^{t_1,t_2}(\fie_i,\fie_j)$ and $a^{t,\tau}(\fie_i,\fie_j)$,
  \begin{align*}
 	a^{t_1,t_2}(\fie_i,\fie_j) &:= \int_{t_1} \int_{t_2} \frac{[\varphi_i(x) - \varphi_i(y)]\,[\varphi_j(x)-\varphi_j(y)]}{ |x-y|^{d+2s}} \ud x\ud y, \\
	a^{t,\tau}(\fie_i,\fie_j) &:= \int_{t} \fie_i(x)\,\fie_j(x) \int_{\tau} \frac{(y-x)^T\, n_\tau}{|x-y|^{d+2s}} \ud s_y \ud x,
\end{align*}
and their approximations. The Theorems \ref{thm:estimates_int_sing_3d3d} to \ref{thm:estimates_int_sing_3d2d} give the corresponding error estimates.
\end{proof}

Since we know how the integration error affects the bilinear form, we can establish rules for the number of Gaussian points per dimension.

\begin{lemma}\label{lem:gauss_rules}
   The following choice of Gauss quadrature rules
   \[
      n_1  \geq \frac{1}{2} (3+l+s) \frac{\lvert\log ( h) \rvert }{\log (2\rho_1)}\quad \text{ and } \quad
      n_2  \geq \frac{1}{2} (2+l+s) \frac{\lvert\log ( h) \rvert }{\log (2\rho_2)}
   \]
conserves the convergence rate in Theorem~\ref{thm:error_u_h}.
\end{lemma}

\begin{proof}
In \cite{AiCl2018} it is already shown that
 \begin{align*}
    \rVert u - u_h \rVert_{\tilde H^s(\Omega)} \leq C\, \left( h^{l-s}\,  |u|_{ H^l(\Omega)} + (E_1+E_2)\, \norm{\Pi_h u}_{L^2(\Omega)} \right),                                     
 \end{align*}
 where $\Pi_h$ is the Scott-Zhang interpolation operator; see \cite{Scott_Zhang} and \cite{Ciralet10}.
The last step is to choose the number of Gauss points in such a way that the errors $E_1$ and $E_2$ in Theorem \ref{error_biliearform} have the same order of magnitude as the other error, i.e.\ we require $E_1 \leq h^{l-s}$ and $E_2 \leq h^{l-s}$.
\end{proof}

\section{Numerical Results}
The focus of the following numerical tests lies on the verification of the presented error estimates.
First, we focus on the error of numerical integration for singular integrals of the type~\eqref{int_3d3d} and~\eqref{int_3d2d}. 
For this purpose, we verify for two selected singularity cases whether the error estimates from the Theorems~\ref{thm:estimates_int_sing_3d3d} and~\ref{thm:estimates_int_sing_3d2d} hold true. Second, we consider a fractional diffusion process and validate the tensor Gauss quadrature rule presented in Lemma~\ref{lem:gauss_rules}. 

\subsection{Singular integrals}
We consider the following integrals of the type~\eqref{int_3d3d} and~\eqref{int_3d2d}
\begin{align*}
    I_1 :=-\int_{t_1 } \int_{t_2} \dfrac{  \varphi_i(y) \,\varphi_j(x)} { \lvert x-y \rvert^{3+2s}}\ud y \ud x
     \quad\text{ and } \quad 
    I_2 :=\int_t \varphi_i^2(x) \int_\tau   \dfrac{ (y-x)^T n_\tau(y) } { \lvert x - y \rvert^{3+2s}}\ud s_y \ud x
\end{align*}
with $s=0.8$.
For the first integral we choose $t_1$ and $t_2$ such that both tetrahedra share a common face~$F$
and for the second integral $t$ and $\tau$ are selected so that the tetrahedron and the panel have a common edge.
For the linear basis functions, we choose $\suppa \fie_i \cap\, \suppa \fie_j = F$ for $I_1$ 
and for $I_2$ we set $i=j$. 
After applying the respective Duffy transformation, a three-dimension integral must be computed numerically in both cases.

Since the integrals cannot be integrated analytically, we replace the exact values $I_1 = I[\tk_{1,F}]$ and $I_2 = I[\tk_{2,E}]$
by $Q^{20}[\tk_{1,F}]$ and $Q^{20}[\tk_{2,E}]$, respectively, where $Q^{20}$ denotes the approximation of the integral with $20$ Gauss points per dimension.
In Figure~\ref{fic:singInt}, the error of numerical integration $E^n$, $n=2,\ldots,8$, is represented for different diameters~$h$.
Notice that the error is plotted logarithmically.
\begin{figure}[h]
	\begin{subfigure}[h]{.49\linewidth}
	    \centering \includegraphics[width=\textwidth]{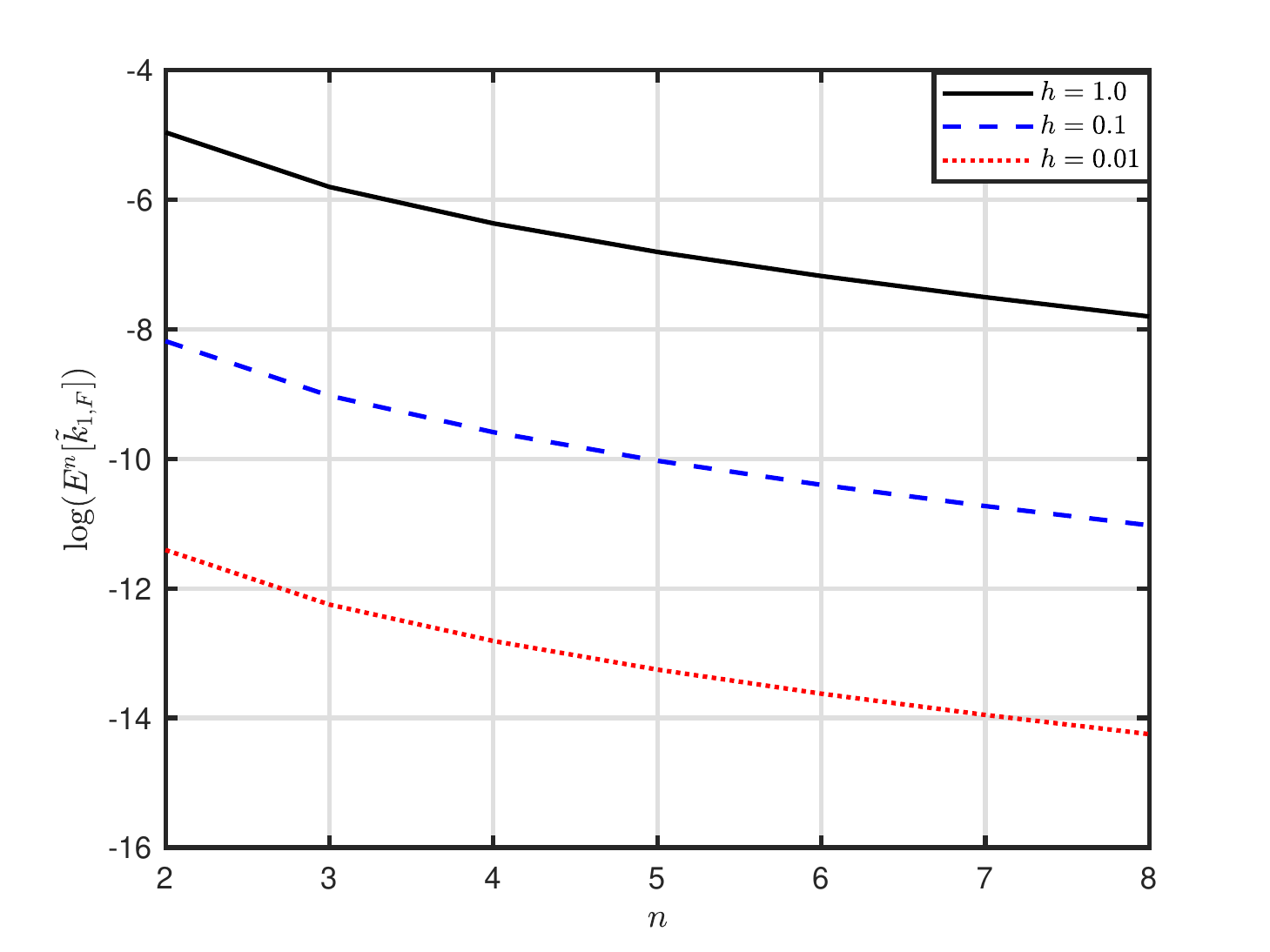}
	    \caption{Integration error for $I_1$} \label{fic:I1}
	\end{subfigure}
	\begin{subfigure}[h]{.49\linewidth}
	    \centering \includegraphics[width=\textwidth]{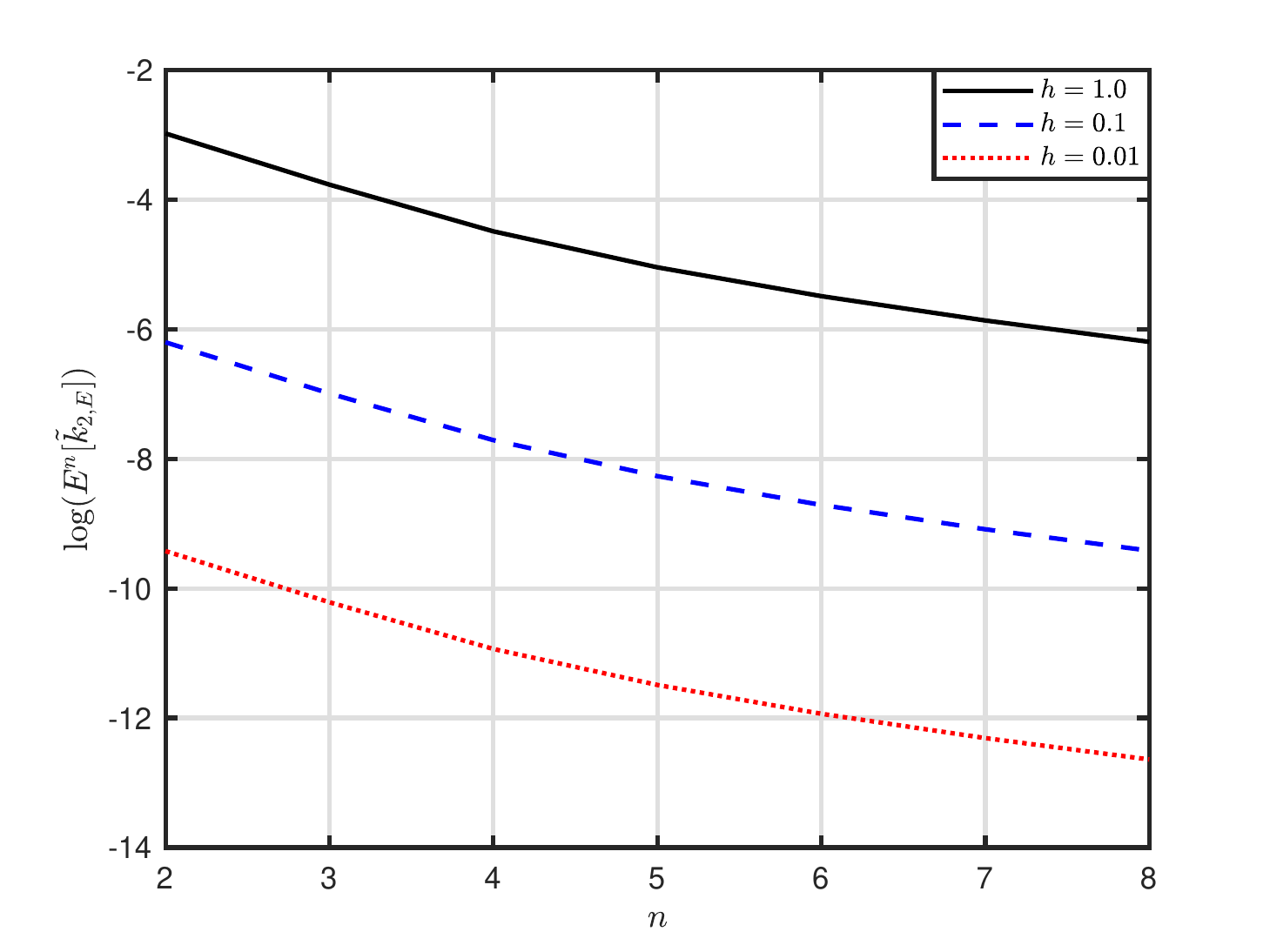}
	    \caption{Integration error for $I_2$} \label{fic:I2}
	\end{subfigure}		
	\caption{integration error for different values of $h$} \label{fic:singInt}
\end{figure}
Based on the Theorems~\ref{thm:estimates_int_sing_3d3d} and~\ref{thm:estimates_int_sing_3d2d}, 
one expects $\log(E^n)$ to be linear with respect to the number of Gauss points.
This can also be observed in the Figures~\ref{fic:I1} and~\ref{fic:I2}. Obviously, $h$ has no influence on the slope of the curves.

\subsection{Fractional Poisson problem}
For the second part of the numerical results, 
we consider the problem~\eqref{eq:fracL} from the beginning of this article, the fractional Poisson problem,
where we set $f=1$ and $\Omega = B_1(0)$:
\begin{align}\label{eq:fracLapBsp} 
\begin{split}
   (-\Delta)^s u &= 1 \quad\text{in } B_1(0), \\
               u &= 0 \quad\text{on } \R^d \backslash B_1(0).
\end{split}               
\end{align}
In this case the analytic solution of the problem 
\[
  u(x) = \frac{2^{-2s}}{\Gamma^2(1+s)} (1-|x|^2)^s,
\]
is well known (see \cite{AcBo2017}), where $\Gamma$ is the Gamma function. The solution~$u$ is depicted in Figure~\ref{fic:us02} and~\ref{fic:us08} for $x_3=0$.
\begin{figure}[htb!]
	\begin{subfigure}[h]{.49\linewidth}
	    \centering \includegraphics[width=\textwidth]{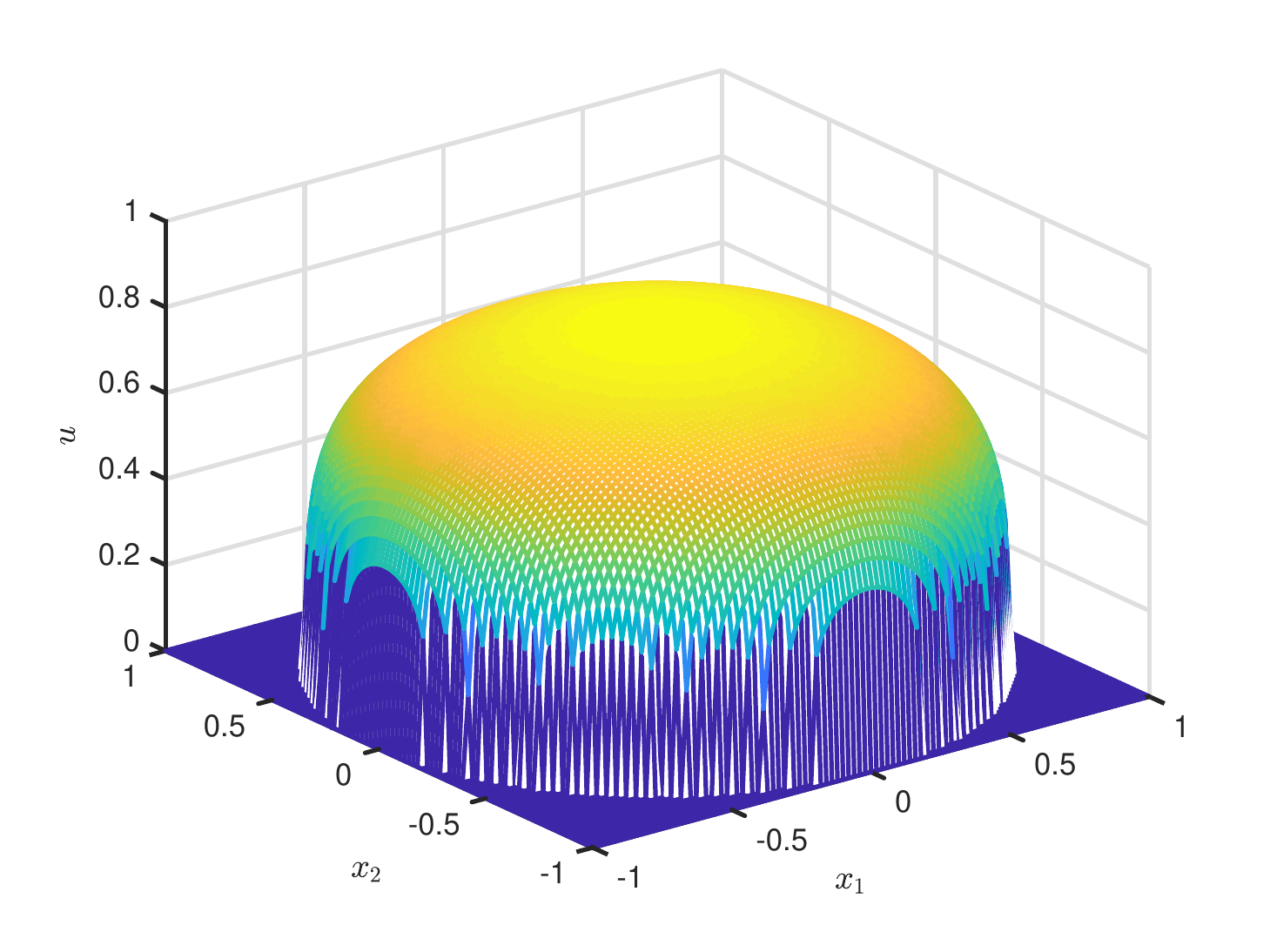}
	    \caption{$s=0.2$} \label{fic:us02}
	\end{subfigure}
	\begin{subfigure}[h]{.49\linewidth}
	    \centering \includegraphics[width=\textwidth]{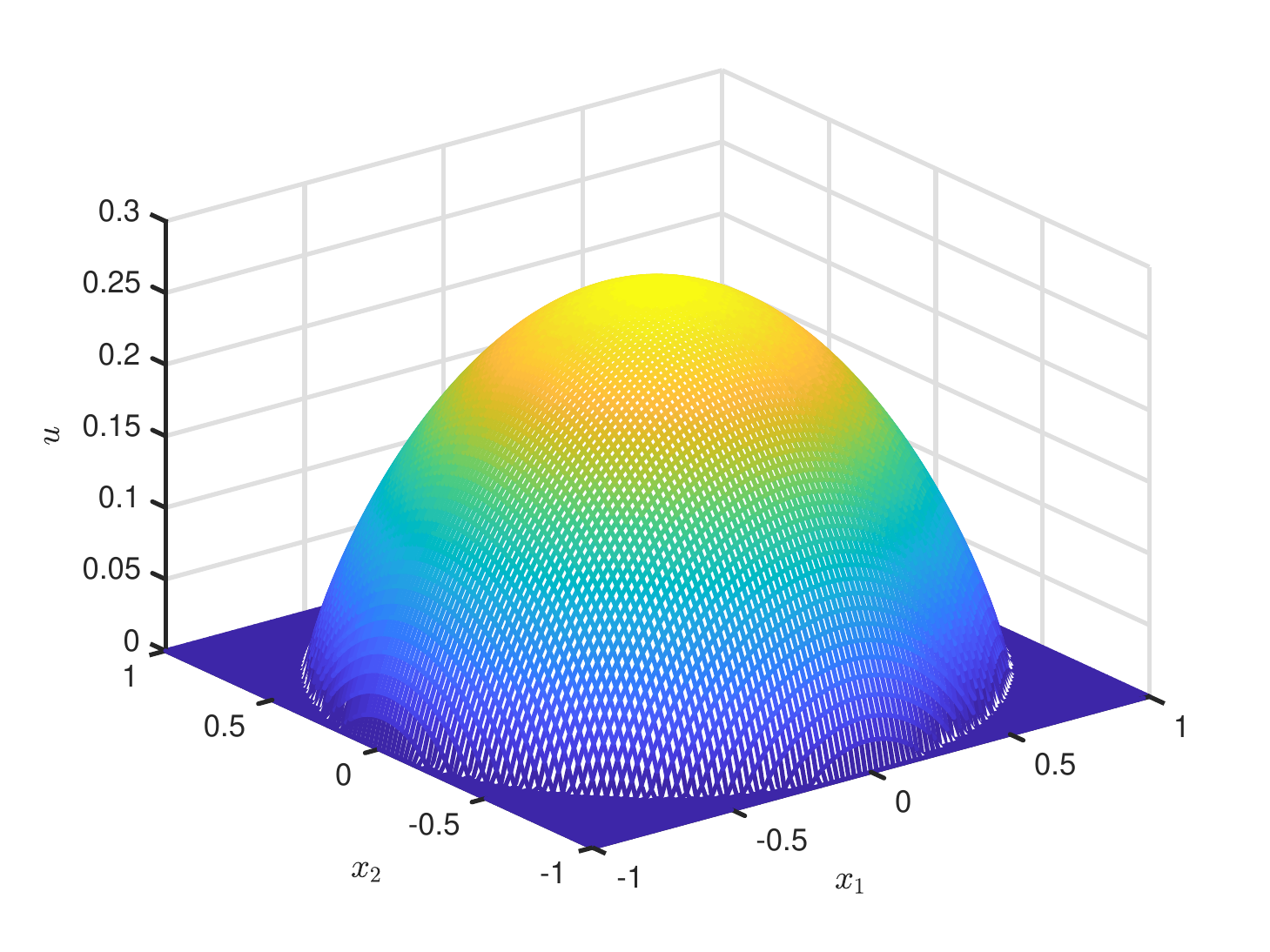}
	    \caption{$s=0.8$} \label{fic:us08}
	\end{subfigure}		
	\caption{Exact solution~$u$} \label{fic:u}
\end{figure}
Following the approach described in Section~\ref{sec:intro}, we discretize the problem~\eqref{eq:fracLapBsp}.
Notice that according to Lemma~\ref{lem:mapTetra} the interior angles of the panels should not be chosen too small,
otherwise the Duffy transformation becomes unstable. 
\begin{figure}[htb!]
	\begin{subfigure}[h]{.49\linewidth}
	    \centering \includegraphics[width=\textwidth]{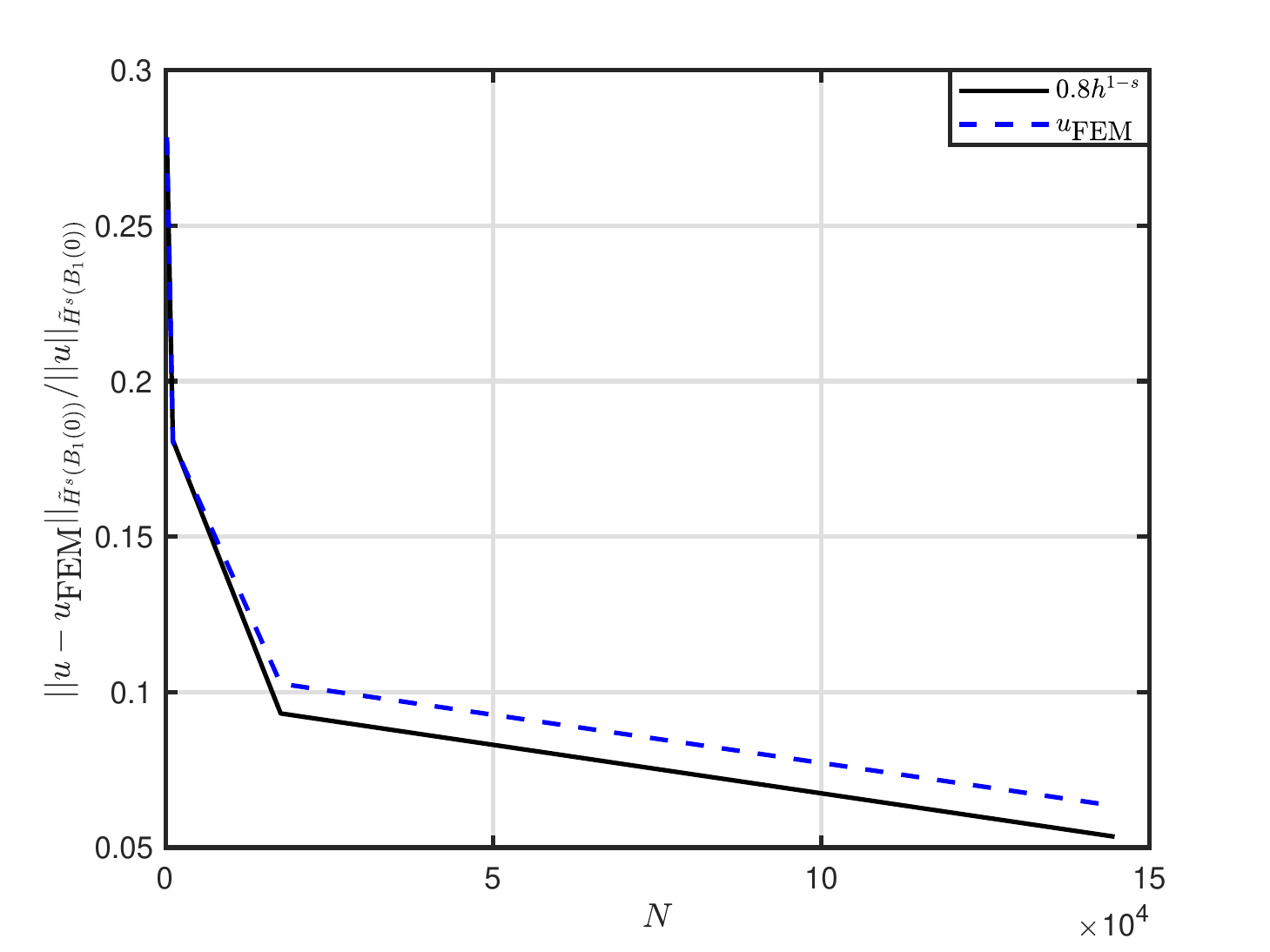}
	    \caption{$s=0.2$} \label{fic:sol_02}
	\end{subfigure}
	\begin{subfigure}[h]{.49\linewidth}
	    \centering \includegraphics[width=\textwidth]{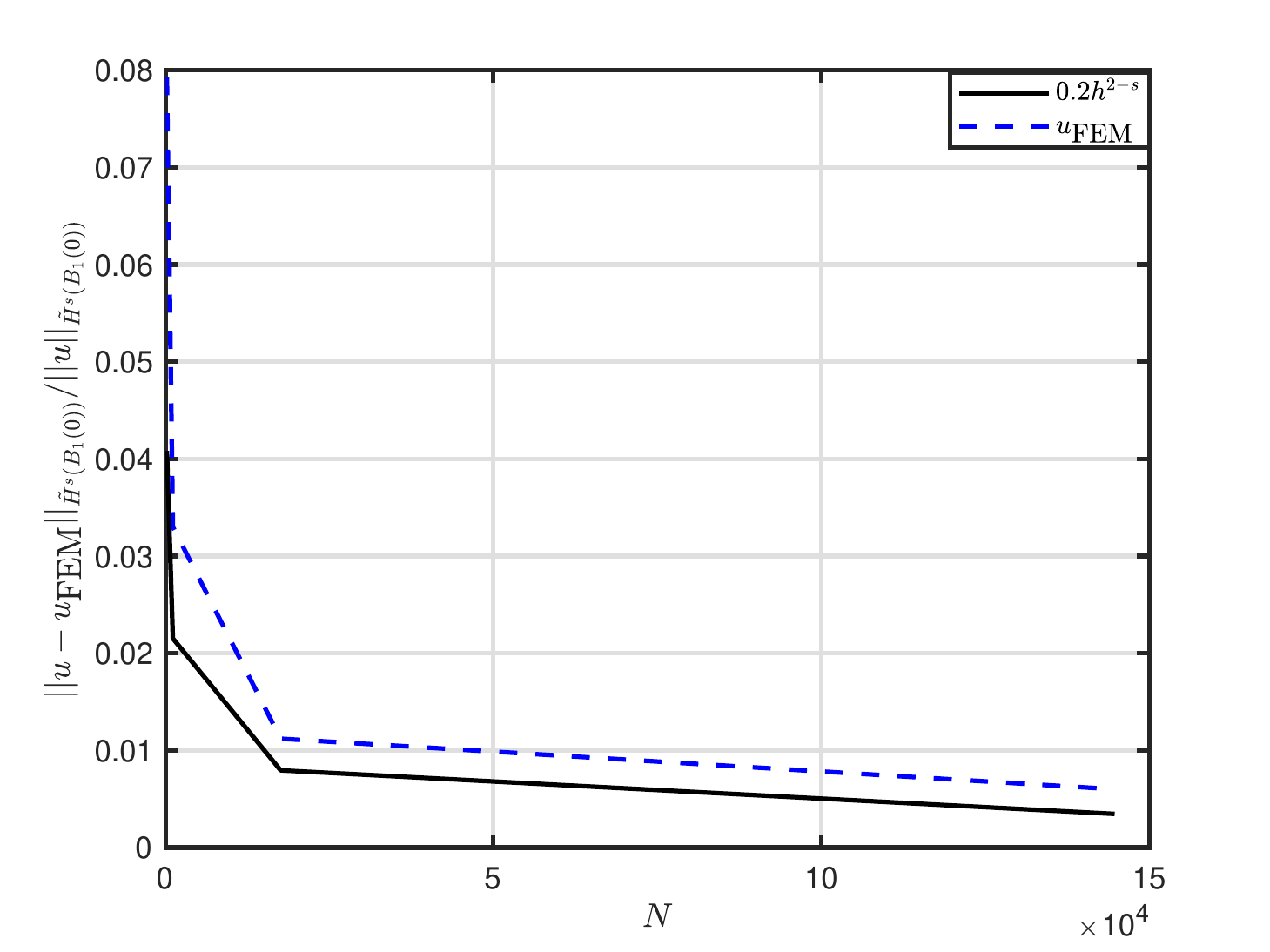}
	    \caption{$s=0.8$} \label{fic:sol_08}
	\end{subfigure}		
	\caption{Relative error between the analytic and numerical solution} \label{fic:sol}
\end{figure}
For the computation of the stiffness matrix we use the lower bound on the number of Gauss points per dimension presented in
Lemma~\ref{lem:gauss_rules}.
Figure~\ref{fic:sol} shows the relative error between the exact solution~$u$ and the finite element solution~$u_{\textnormal{FEM}}$ 
together with the theoretical error estimate of Theorem~\ref{thm:error_u_h}, where we set the smoothness index~$l$ to the maximum value in each case.
In both Figures~\ref{fic:sol_02} and~\ref{fic:sol_08} we see that the relative errors w.r.t.\ the $\tilde H^s(B_1(0))$-norm are close to the error estimate and behave in the same way. 
Comparing Figure~\ref{fic:sol_02} and Figure~\ref{fic:sol_08}, it seems
that the relative errors for $s=0.2$ are an order of magnitude larger than those for $s=0.8$.
This is due to the smoothness of the solution.

\end{document}